\documentclass[11pt]{article}

\usepackage[utf8]{inputenc} 
\usepackage[T1]{fontenc}

\usepackage{amsthm, amsmath, amssymb, amsfonts, url, booktabs, tikz, setspace, fancyhdr, bm}
\usepackage{hyperref}
\usepackage{geometry}
\usepackage{hyperref, enumerate}
\usepackage[shortlabels]{enumitem}
\usepackage[babel]{microtype}
\usepackage[english]{babel}
\usepackage[capitalise]{cleveref}

\usepackage{etoolbox}

\usepackage[linesnumbered, algoruled, lined]{algorithm2e}
\usepackage{bbm,tkz-graph,subcaption}
\usepackage{csquotes}
\usepackage{mathrsfs}
\usepackage{mathabx}
\usetikzlibrary{patterns}
\usetikzlibrary{shapes} 
\usepackage{bbm}

\usepackage{amsfonts}
\usepackage{amscd}
\usepackage{graphicx}
\usepackage{enumitem}
\usepackage{verbatim}
\usepackage{hyperref}
\usepackage{amsmath,caption}
\usepackage{url,pdfpages,xcolor,framed,color}
\usepackage{todonotes}

\usepackage{thmtools}

\usepackage{colortbl}
\usepackage{soul}

\newtheorem{thm}{Theorem}[section]

\newtheorem{obs}[thm]{Observation}
\newtheorem{p}[thm]{Problem}

\theoremstyle{definition}
\newtheorem{defi}[thm]{Definition}
\newtheorem*{defi*}{Definition}

\newtheorem{examp}[thm]{Example}

\newtheorem{constr}[thm]{Construction}
\Crefname{constr}{Construction}{Constructions}

\Crefname{thm}{Theorem}{Theorems}
\Crefname{thmenumi}{Theorem}{Theorems}
\AtBeginEnvironment{thm}{%
    \crefalias{enumi}{thmenumi}%
    \setlist[enumerate,1]{
        label={\textit{(\roman*)}},
        ref={\thethm.(\roman*)}
    }%
}

\newcommand*{\myproofname}{Proof}

\crefname{algorithm}{Algorithm}{Algorithms}

\newcommand{\V}{V}
\newcommand{\E}{E}
\newcommand{\Edel}{\E_{\textrm{del}}}
\newcommand{\G}{G}
\newcommand{\Greg}{\G_{\textrm{reg}}}

\newcommand{\Gbireg}{\G_{\textrm{bireg}}}

\newcommand{\T}{T}

\newcommand{\NMTree}[4]{\T_{#1,#2,#3}^{(#4)}}
\newcommand{\NMTreeFam}[4]{\mathcal{T}_{#1,#2,#3}^{(#4)}}
\newcommand{\rg}{(r,g)}
\newcommand{\rmg}{(\{r,m\};g)}
\newcommand{\rmgIn}[3]{(\{#1,#2\};#3)}

\newcommand{\rmgnFam}[4]{\mathcal{G}_{#1,#2,#3,#4}}
\newcommand{\biregEven}{(\{r,m\};2t)}
\newcommand{\size}[1]{\lvert #1 \rvert}
\newcommand{\ceil}[1]{\left\lceil #1 \right\rceil}
\newcommand{\powerSum}[3]{\ensuremath{\sum_{i=#1}^{#2} (#3)^i}}

\newcommand{\roott}{\mathsf{w}}
\newcommand{\vroot}{\ensuremath{v_{\roott}}}
\newcommand{\uroot}{\ensuremath{u_{\roott}}}

\newcommand{\edge}[2]{\{ #1, #2 \}} 
\renewcommand{\degree}[1]{\textrm{deg}({#1})}
\newcommand{\dist}[2]{\textrm{dist}(#1,#2)}
\newcommand{\distwG}[3]{\textrm{dist}_{#1}(#2,#3)}

\newcommand{\dmin}{d_{\text{min}}}
\newcommand{\floor}[1]{\left\lfloor #1 \right\rfloor}

\newcommand{\cyc}{\mathcal{C}}

\newcommand{\comb}[1]{\left( #1 \right)}
\newcommand{\setCombs}{\mathfrak{P}}
\newcommand{\cmax}{\placement_{\textrm{max}}}
\newcommand{\placement}{\mathfrak{p}}

\newcommand{\algoName}[1]{\emph{#1}}

\newcommand{\refAlgoLine}[1]{{\footnotesize{\textbf{\ref{#1}}}}}

\newcommand{\nauty}{\texttt{nauty}}
\newcommand{\geng}{\texttt{geng}}
\newcommand{\multigraph}{\texttt{multigraph}}
\newcommand{\biregGen}{\texttt{biregGen}}

\newcommand{\numImprovedUpBounds}{49}
\newcommand{\numExhLists}{24}
\newcommand{\numImprovedUpBoundsThm}{73}
\newcommand{\numImprovedUpBoundsTotal}{122}

\title{Computational methods for finding bi-regular cages}

\SetArgSty{}	
\SetKwComment{Comment}{// }{}

\SetCommentSty{mycommfont}
\SetKwIF{If}{ElseIf}{Else}{if}{then}{else if}{else}{}
\SetKwFor{For}{for}{do}{}
\SetKwFor{ForAll}{forall}{do}{}
\SetKwFor{ForEach}{foreach}{do}{}
\SetKwFor{While}{while}{do}{}

\newcommand{\True}{\AlCapSty{true}}
\newcommand{\False}{\AlCapSty{false}}

\providecommand{\keywords}[1]
{
  \small	
  \textbf{\textit{Keywords---}} #1
}

\providecommand{\msc}[1]
{
  \small	
  \textbf{\textit{Math. Subj. Class. (2020)---}} #1
}

\date{}
\author{
Jan Goedgebeur \thanks{Department of Computer Science, KU Leuven Campus Kulak-Kortrijk, Etienne Sabbelaan 53, 8500 Kortrijk, Belgium. 
\protect\href{mailto:jan.goedgebeur@kuleuven.be}{\protect\nolinkurl{jan.goedgebeur@kuleuven.be}},
\protect\href{mailto:jorik.jooken@kuleuven.be}{\protect\nolinkurl{jorik.jooken@kuleuven.be}},
 and \protect\href{mailto:tibo.vandeneede@kuleuven.be}{\protect\nolinkurl{tibo.vandeneede@kuleuven.be}}} 
 \thanks{Department of Applied Mathematics, Computer Science and Statistics, Ghent University, Krijgslaan 281-S9, 9000 Ghent, Belgium.}
\and Jorik Jooken
 \footnotemark[1] 
\and Tibo Van den Eede
 \footnotemark[1]
}

\begin{document}
\maketitle

\begin{abstract}
An \textit{$\rmg$-graph} is a (simple, undirected) graph of girth $g\geq3$ with vertices of degrees $r$ and $m$ where $2 \leq r < m$ . Given $r,m,g$, we seek the $\rmg$-graphs of minimum order, called \textit{$\rmg$-cages} or \textit{bi-regular cages}, whose order is denoted by $n\rmg$. In this paper, we use computational methods for finding $\rmg$-graphs of small order. Firstly, we present an exhaustive generation algorithm, which leads to -- previously unknown -- exhaustive lists of $\rmg$-cages for $\numExhLists$ different triples $(r,m,g)$. This also leads to the improvement of the lower bound of $n\rmgIn{4}{5}{7}$ from 66 to 69. Secondly, we improve $\numImprovedUpBounds$ upper bounds of $n\rmg$ 
based on constructions that start from $r$-regular graphs. Lastly, we generalize a theorem by Aguilar, Araujo-Pardo and Berman [arXiv:2305.03290, 2023], leading to $\numImprovedUpBoundsThm$ additional improved upper bounds. 
    
\end{abstract}

\keywords{Graph algorithms, Extremal problems, Cage Problem, Degree sequences.}

\vspace*{1pt}

\msc{05C07, 05C35, 05C85, 68R10, 90C35}

\section{Introduction}

The Cage Problem is a renowned problem in extremal graph theory and revolves around finding (the order of) a specific type of graphs called \textit{cages}. An \textit{$\rg$-graph} is a simple undirected $r$-regular graph of girth $g$, where the latter means that its shortest cycle has length $g$. An \textit{$\rg$-cage}, or simply \textit{cage}, is an $\rg$-graph of minimum order. The order of an $\rg$-cage is denoted by $n\rg$. Cages were first introduced in 1947 by Tutte~\cite{Tutte1947}. In 1963 Sachs proved that for all $r \geq 2$ and $g \geq 3$ an $\rg$-graph always exists~\cite{Sachs1963}. Nevertheless, $\rg$-cages were only determined for a limited number of pairs $\rg$ as this is a notoriously hard problem. For pairs $\rg$ where $\rg$-cages are unknown, only lower and upper bounds for $n\rg$ have been established. For an overview of the known cages and bounds, we refer to \cite{CageSurvey}. We now mention a well-known lower bound, called the \textit{Moore bound} $M\rg$, for $n\rg$ (which holds for any $r \geq 2$ and $g \geq 3$).
\begin{equation}\label{eq:Mbound}
	n\rg \geq M(r,g) =
	\begin{cases}
		 1 + r \powerSum{0}{t-1}{r-1} & \text{for} \: g=2t+1 \\
		 2 \powerSum{0}{t-1}{r-1} & \text{for} \: g=2t . \\
	\end{cases}
\end{equation}
This bound follows from constructing a tree, called the \textit{Moore tree}, with all internal vertices of degree $r$ until depth $\floor{(g-1)/2}$. For odd $g$, the tree has a vertex as root and for even $g$ we consider an edge as the root. This tree occurs as a subgraph of any $(r,g)$-graph.

Due to the difficulty of the original Cage Problem, researchers started looking at variants of cages. One of the most studied variants is the bi-regular cage. 
An \textit{$\rmg$-graph} is a graph with degree set $\{r,m\}$ where $r < m$ and girth $g$. Analogously, an \textit{$\rmg$-cage} or \textit{bi-regular cage} is an $\rmg$-graph of minimum order and this order is denoted by $n\rmg$. The existence of $\rmg$-graphs for $2\leq r<m$ and $g \geq 3$ was proven in 1981 by Chartrand, Gould and Kapoor~\cite{Chartrand1981}. 
Furthermore, there is a lower bound $M\rmg$ for $n\rmg$ for any $2 \leq r < m$ and $g \geq 3$ similar to the Moore bound for $n\rg$. The only difference is that we replace a vertex of degree $r$ by a vertex of degree $m$ in the root of the Moore tree. 

\begin{equation}\label{eq:Mboundbireg}
	n\rmg \geq M\rmg =
	\begin{cases}
		1 + m \powerSum{0}{t-1}{r-1} & \text{for} \: g=2t+1 \\
		1 + m \powerSum{0}{t-2}{r-1} + (r-1)^{t-1} & \text{for} \: g=2t . \\
	\end{cases}
\end{equation}

We note that $\rmgIn{r}{m}{3}$ \cite{Kapoor1977}, $\rmgIn{2}{m}{g}$ and $\rmgIn{r}{m}{4}$ \cite{Chartrand1981} are considered to be the easier cases, as there bi-regular cages are known for all $r,m,g$. A considerable amount of research has been done on finding bi-regular cages and improving lower and upper bounds for $n\rmg$, including~\cite{biregGirth5, semiCubic2023, Aurajo-Pardo2007, biregGirth8, ConstructionsBireg, biregEven, semiRegCages, Dn-cages, biregOdd, Hanson1992, Yang2003}. In particular, Exoo and Jacjay proved in \cite{biregOdd} that for every $r\geq3$ and odd $g\geq3$ there exists an $m_0$ such that for every even $m\geq m_0$, it holds that $n\rmg=M\rmg$, thereby solving the problem asymptotically for odd girth. Moreover, when $r$ is odd, the restriction on the parity of $m$ can be removed. However, $m_0$ is only known for a very limited number of pairs $(r,m)$. On top of that, this result leaves $n\rmg$ for $m<m_0$ unresolved. Besides this, in contrast with odd girth, it holds that  $n\rmg>M\rmg$ for even $g \geq 6$ \cite{Aurajo-Pardo2007,Yang2003}. Hence, many orders of bi-regular cages are still unknown and in even more cases the complete lists of such graphs have not been determined.

Furthermore, most of the papers on bi-regular cages do not mention the use of computers. This contrasts with the original Cage Problem where computer methods led to finding multiple $\rg$-cages and improving lower and upper bounds for $n\rg$ in, among others, \cite{BRAY2000,Brink1995,Conder1997SmallTG,Exoo2002AST,Exoo2011,McKay1998}. Motivated by these results, we introduce computer(-aided) methods in this paper, leading to -- previously unknown -- exhaustive lists of bi-regular cages for $\numExhLists$ different triples $(r,m,g)$ and improving lower and upper bounds for $\numImprovedUpBoundsTotal$ triples within the ranges we consider.

\subsection{Outline of the paper}\label{subsec:outline}

The paper consists of three main sections, each presenting a different method for finding $\rmg$-graphs and describing its results. The code and data related to the computer methods we developed for this paper is available at \url{https://github.com/tiboat/biregGirthGraphs} \cite{repo}. \cref{sec:exh_gen} presents a backtracking algorithm to exhaustively generate $\rmg$-graphs of a given order. The pruning rules which make this algorithm efficient are obtained from theoretical results which we also present in this section. 
With this algorithm we generated $\numExhLists$ exhaustive lists of $\rmg$-cages and improved the lower bound of $n\rmgIn{4}{5}{7}$ from 66 to 69. \cref{sec:constr} describes constructions for $\rmg$-graphs based on appropriate modification of $r$-regular graphs. Using existing lists of $r$-regular graphs as seeds we were able to improve $\numImprovedUpBounds$ upper bounds on $n\rmg$. Next, \cref{sec:gen_up_bound} generalizes a construction of $\rmg$-graphs out of $\rg$-graphs by Aguilar, Araujo-Pardo and Berman \cite{semiCubic2023}, also leading to improved upper bounds. \cref{sec:conc} presents some concluding remarks and future work.

\cref{sec:sanity} describes extra steps that we took to ensure the correctness of the implementation of the exhaustive generation algorithm presented in \cref{sec:exh_gen}.
\cref{sec:tables} summarizes the best known lower and upper bounds for $n\rmg$ that we found in the literature, augmented with the bounds we were able to improve in this paper. The summary limits $r,m,g$ to the ranges for which we were able to improve bounds with our methods from \cref{sec:exh_gen,sec:constr}.

\subsection{Notation and definitions}\label{subsec:not&def}

The set of vertices and set of edges of a graph $\G$ are denoted by $\V(\G)$ and $\E(\G)$, respectively. An edge between two vertices $v$ and $w$ will be denoted by $\edge{v}{w}$. The degree of a vertex $v$ is denoted by $\degree{v}$. We call a vertex $v$ \textit{isolated} if $\degree{v} = 0$. The girth of a graph $\G$, which is the length of the shortest cycle in $\G$, will be denoted by $g(\G)$. The distance between two vertices $v_1,v_2$ in graph $\G$ is denoted by $\distwG{\G}{v_1}{v_2}$ or simply $\dist{v_1}{v_2}$ if the graph in which we define this distance is clear from the context. We define the distance between a vertex $u$ and an edge $\edge{v}{w}$ to be $\dist{u}{\edge{v}{w}} = \textrm{min} \left( \dist{u}{v}, \dist{u}{w} \right)$ and define the distance between two edges $\edge{v_1}{w_1}, \edge{v_2}{w_2}$ to be $\dist{\edge{v_1}{w_1}}{\edge{v_2}{w_2}} = \textrm{min} \left( \dist{v_1}{\edge{v_2}{w_2}}, \dist{w_1}{\edge{v_2}{w_2}} \right)$. In this paper a root $\roott$ of a tree can either be a vertex or an edge. Suppose $T$ is a tree with root $\roott$, then we denote the set of vertices at distance $d$ from $\roott$ as $L_d$. We say that the vertices of $L_d$ are at \textit{level} $d$. The subset $L_{d,k}$ of $L_d$ denotes the set of vertices at distance $d$ from $\roott$ which have degree $k$. The set of vertices of a graph of degree $k$ are denoted by $V_k$. Finally, we denote the set of pairwise non-isomorphic $\rmg$-graphs of order $n$ as $\rmgnFam{r}{m}{g}{n}$.

\section[Exhaustive generation of ({r,m};g)-graphs]{Exhaustive generation of $\rmg$-graphs}\label{sec:exh_gen}

To exhaustively generate $\rmg$-graphs we make use of a backtracking algorithm. This algorithm is based on the one for $\rg$-graphs by McKay, Myrvold and Nadon~\cite{McKay1998}, extended with the pruning rule from~\cite{Exoo2011}. To generate $\rmg$-graphs we made several adaptations and found new pruning rules, which will be explained in this section. 
The algorithm from~\cite{McKay1998} starts with the Moore tree and recursively adds edges to it. For $\rmg$-graphs we start with what we will call a \textit{bi-regular Moore tree}. The order of this tree corresponds to the bi-regular Moore bound as mentioned in \cref{eq:Mboundbireg}. The main algorithmic difficulty in comparison to $\rg$-graphs is (of course) that a vertex can be of either degree $r$ or $m$ instead of only one degree $r$. The new pruning rules limit which vertices cannot be of degree $m$ and hence need to be of degree $r$. Before discussing the pruning rules, we first discuss a refined lower bound for $n\rmg$ in \cref{subsec:refined_low_bound}, which will be useful for the pruning rules as described in \cref{subsec:distDegM,subsec:mPlacement}. Next, we present the backtracking algorithm in \cref{sec:biregBacktrackAlgo} and lastly, we discuss the results we obtained with the algorithm in \cref{subsec:exh_gen_exp}.

\subsection[Refined lower bound for n({r,m};g)]{Refined lower bound for $n\rmg$}\label{subsec:refined_low_bound}

First, we explicitly define a bi-regular Moore tree, which we will abbreviate as BMT.

\begin{defi}[Bi-regular Moore tree (BMT)]\label{def:mTree}
	A tree with root  $\roott$ is a bi-regular Moore tree $\NMTree{r}{m}{d}{\roott}$ if 
	\begin{itemize}[noitemsep,topsep=0pt]
		\item $r < m$,
		\item $\degree{\vroot} = m$ with $\roott = \vroot$ or $\roott = \edge{\uroot}{\vroot}$,
		\item all internal vertices (non-leaves) have degree $r$ or $m$, and 
		\item the distance between $\roott$ and all leaves equals $d$.
	\end{itemize}
\end{defi}

\begin{examp}[Bi-regular Moore tree (BMT)]\label{ex:bmt}
	\cref{fig:m_plaatsing} shows a BMT $\NMTree{3}{4}{3}{\vroot}$ with root $\vroot$ of degree $4$.\\
	
	\centering
	\includegraphics[page=3, scale=1.1]{ipe/mooreTrees.pdf}
	\captionof{figure}{A BMT $\NMTree{3}{4}{3}{\vroot}$.}
	\label{fig:m_plaatsing}
\end{examp}

In case only one vertex in the root of the BMT is of degree $m$, then the BMT corresponds to the tree associated with the bi-regular Moore bound (\cref{eq:Mboundbireg}) where $d=t$, $\roott = \edge{\uroot}{\vroot}$ for $g=2t$ and $d=t-1$, $\roott = \vroot$ for $g=2t-1$. In this subsection we present a lower bound on the amount of vertices of an $\rmg$-graph given the amount of vertices of degree $m$ on each level of the bi-regular Moore tree. With the notation that we introduced this means that we assume we know $\size{L_{i,m}}$ for $i \in \{0,\ldots,d-1\}$. The fundamental, simple observation that will lead to the lower bound is the following.

\begin{obs}\label{obs:baseMtree}
	Suppose $\NMTree{r}{m}{d}{\roott}$ is a BMT and $3 \leq r < m$. Then for all $j \in \{2, \ldots, d\}$ it holds that
	\begin{equation*}
		\size{L_j} = (m-r) \size{L_{j-1,m}} + (r-1) \size{L_{j-1}}.
	\end{equation*}
\end{obs}

\begin{proof}
    The amount of vertices at level $j$, $\size{L_j}$, is equal to the amount of edges between level $j-1$ and level $j$. There are at least $r-1$ edges for each vertex in $L_{j-1}$. For each vertex of degree $m$ in $L_{j-1}$ we have $m-r$ edges on top of that.
\end{proof}

By rewriting $\size{L_j}$ in a non-recursive form (using \cref{obs:baseMtree}) and summing the amount of vertices on each level we obtain the more refined lower bound, presented in \cref{thm:lowBoundMtree}. A detailed derivation can be found in the Master's thesis of the third author~\cite[Section~3.1]{thesis}. 

\begin{thm}\label{thm:lowBoundMtree}
	Suppose $\NMTree{r}{m}{d}{\roott}$ is a BMT, $3 \leq r < m$ and $d \geq 2$. Then
	\begin{equation*}
		\size{\V(\NMTree{r}{m}{d}{\roott})} = \size{L_0} + \size{L_1} \powerSum{0}{d-1}{r-1} + (m-r) \left(  \sum_{j=1}^{d-1}  \size{L_{j,m}} \powerSum{0}{d-1-j}{r-1} \right).
	\end{equation*}
\end{thm}

\cref{thm:lowBoundMtree} thus gives the order of a BMT, given the amount of vertices of degree $m$ on each level. 
The exhaustive generation algorithm will generate $\rmg$-graphs for a specific order $n$. With \cref{thm:lowBoundMtree} we can derive which combinations of vertices of degree $m$ on the levels of the BMT lead to the BMT's order being larger than $n$. Hence, we know these combinations are impossible, which will be useful for the pruning rules in the following two subsections.

\subsection[Pruning rule: distance between vertices of degree m]{Pruning rule: distance between vertices of degree $m$}\label{subsec:distDegM}
Araujo-Pardo, Exoo and Jacjay proved lower bounds for $n\biregEven$ by making a distinction based on the amount of vertices of degree $m$, denoted by $\size{\V_m}$, at a minimum distance $\dmin$ from each other~\cite{biregEven}.
Among others, they provide lower bounds for $n\biregEven$ with $3 \leq r < m$, $t>3$, $\size{\V_m} \geq 2$ under different cases for $\dmin$, presented in \cref{thm:distAll}.

\begin{thm}[{\cite[Lemmas~3.2~and~3.3]{biregEven}} and {\cite[Case~2 of Theorem~2.3]{Aurajo-Pardo2007}}]\label{thm:distAll}
If $t \geq 4$, $3 \leq r < m$, and $\G$ is an $\biregEven$-graph with $\size{\V_m} \geq 2$. Then the following hold.

\begin{enumerate}
    \item If $\dmin = 1$, then
	\begin{equation*}
		\size{\V(\G)} \geq 2 + 2(m-1) \dfrac{(r-1)^{t-1}-1}{r-2}.
	\end{equation*}\label{thm:dist1}

    \item If $2 \leq \dmin \leq t-1$, then 
	\begin{equation*}
		\size{\V(\G)} \geq 1 + m \left( \powerSum{0}{t-2}{r-1}\right) + (m-r) \left( \powerSum{0}{t-1-\dmin}{r-1} \right) + (r-1)^{t-1} .
	\end{equation*}\label{thm:dist2}

    \item If $\dmin \geq 3$, then
	\begin{equation*}
		\size{\V(\G)} \geq 1 + m \dfrac{(r-1)^{t-1}-1}{r-2} + (r-1)^{t-1} + \dfrac{(m-r)(r-2)(r-1)^{t-2}}{r} .
	\end{equation*}\label{thm:dist3}
\end{enumerate}
\end{thm}

Note that for \cref{thm:dist2} the lower bound decreases for increasing $\dmin$.
\cref{thm:distAll} can be used for pruning as follows: given $r,m,g$ and the order $n$, the smallest $\dmin$ can be calculated such that the lower bound is at most $n$. Thus, when adding an edge in the algorithm, we can check whether vertices that have degree larger than $r$, are at a distance of at least $\dmin$ from each other (since all vertices of degree larger than $r$ must eventually obtain degree $m$).

The downside of \cref{thm:distAll} is that it is restricted to even girth $g \geq 8$ and hence does not cover odd $g$ and even $g < 8$.
However, using \cref{thm:lowBoundMtree} we can also derive a lower bound for odd $g$ and even $g\geq4$. This can be done by taking $\size{L_{\dmin,m}} = 1$ and $\size{L_{j,m}}=0$ for all $j \in \{1, \ldots,  \dmin-1, \dmin+1, \ldots, t-1\}$, which means that there are exactly two vertices of degree $m$ at distance $\dmin$ in the BMT. On top of that, for even $g \geq 8$ \cref{thm:lowBoundMtree} is equivalent to cases (i) and (ii) from \cref{thm:distAll}. Thus, in other words, \cref{thm:lowBoundMtree} is a generalization of these two cases.

\subsection[Pruning rule: amount of vertices of degree m in a BMT]{Pruning rule: amount of vertices of degree $m$ in a BMT}\label{subsec:mPlacement}
For the next pruning rule, we make use of \cref{thm:lowBoundMtree} to limit the amount of vertices of degree $m$ on the different levels in the BMT. We define the sequence of vertices of degree $m$ in a BMT as an \textit{$m$-placement}.

\begin{defi}[$m$-placement]\label{def:m_plaatsing}
	The $m$-placement of a BMT $\NMTree{r}{m}{d}{\roott}$ is the $d$-tuple containing the amount of vertices of degree $m$ from level 0 up to and including $d-1$ in $\NMTree{r}{m}{d}{\roott}$, denoted as
	\begin{equation*}
		\placement(\NMTree{r}{m}{d}{\roott}) = \comb{ \size{L_{0,m}}, \size{L_{1,m}}, \ldots,  \size{L_{d-1,m}}}.
	\end{equation*}
    The suffix $[i]$ indicates the $(i+1)$-th element of an $m$-placement: $\placement(\NMTree{r}{m}{d}{\roott})[i] = \size{L_{i,m}}$.
\end{defi}

\begin{examp}[$m$-placement]\label{ex:m_plaatsing}
	\cref{fig:m_plaatsing} shows a BMT $\NMTree{3}{4}{3}{\vroot}$ with $m$-placement $\placement(\NMTree{3}{4}{3}{\vroot})=\comb{1,1,3}$. Furthermore, it holds that $\placement(\NMTree{3}{4}{3}{\vroot})[0]=1$, $\placement(\NMTree{3}{4}{3}{\vroot})[1]=1$ and $\placement(\NMTree{3}{4}{3}{\vroot})[2]=3$.
\end{examp}

Suppose one wants to generate all $\rmgIn{3}{4}{7}$-graphs of order $n=35$. Then there are multiple BMTs of order at most $n$ with different $m$-placements. 
Suppose $\NMTreeFam{3}{4}{3}{\vroot}$ is the set of all BMTs for $\rmgIn{3}{4}{7}$ with order at most $n=35$ and suppose $\setCombs = \{\placement(\NMTree{3}{4}{3}{\vroot}) \mid \NMTree{3}{4}{3}{\vroot} \in \NMTreeFam{3}{4}{3}{\vroot} \}$ is the set of all $m$-placements of BMTs in $\NMTreeFam{3}{4}{3}{\vroot}$. Then, by using \cref{thm:lowBoundMtree}, we can determine that $\setCombs = \{  \comb{ 1, 0, 0}, \ldots, \comb{ 1, 0, 6 },  \comb{ 1, 1, 0 }, \ldots,$ $\comb{ 1, 1, 3 }, \comb{ 1, 2, 0 }  \}$. The amount of vertices of degree $m$ on level $0$ is  always equal to 1 for odd $g$. Yet, we can derive from $\setCombs$ that $\size{L_{1,m}}$ is at most equal to 2 and $\size{L_{2,m}}$ at most equal to 6. We will call the $m$-placement that takes the maximum amount of vertices on each level the \textit{maximal $m$-placement}.

\begin{defi}[Maximal $m$-placement]
	Suppose $\NMTreeFam{r}{m}{d}{\roott}$ is a set of BMTs and $\setCombs = \{\placement(\NMTree{r}{m}{d}{\roott}) \mid \NMTree{r}{m}{d}{\roott} \in \NMTreeFam{r}{m}{d}{\roott} \}$. 
	Then the maximal $m$-placement of $\NMTreeFam{r}{m}{d}{\roott}$ is equal to
	\begin{equation*}
		\cmax =  \left( \underset{\placement \in \setCombs}{\max} \: \placement[0], \ldots, \underset{\placement \in \setCombs}{\max} \: \placement[d-1]  \right)  .
	\end{equation*}	
\end{defi}

Using the maximal $m$-placement we can derive the following pruning rule: suppose $\cmax$ is the maximal $m$-placement of the set of possible BMTs $\NMTreeFam{r}{m}{d}{\roott}$ for given $r,m,g$ and $n$. If the amount of vertices of the current constructed graph on level $i \in \{0, \ldots, d-1\}$ of degree larger than $r$ equals $\cmax[i]$, then no edges need to be added to vertices of level $i$ that have degree $r$.

\subsection{The backtracking algorithm}\label{sec:biregBacktrackAlgo}

\newcommand{\searchDepth}{d_{s}}
\newcommand{\currDeg}{deg_{curr}}
\newcommand{\eersteFase}{firstPhase}
\newcommand{\validEdges}{E_{\textrm{add}}}
\newcommand{\backtrack}{recursivelyAddEdges}
\newcommand{\checkVisited}{GWasUpdated}

In this subsection we will present pseudocode of the backtracking algorithm to exhaustively generate all pairwise non-isomorphic $\rmg$-graphs of a given order $n$. Algorithm~\ref{alg:genBiregGirthGraphs} fulfills this task. It calls the method $\backtrack$ (shown in Algorithm~\ref{alg:backtrack2}) in which edges are recursively added one by one. The algorithm shares similarities with the algorithms discussed in \cite{exhEgr,vertexGirthRegular}. For the rest of this section we will give some clarifications regarding the pseudocode. For more details on the implementation, such as the data structures we used, we refer to~\cite{thesis} and the actual code in~\cite{repo}.

\begin{itemize}[noitemsep]
    \item The set $\validEdges$ contains edges which are not part of the current constructed graph, but are eligible to be added to the graph. An edge is \textit{eligible} as long as it is not detected to be prunable, because adding it would lead to a vertex having degree larger than $m$, lead to a cycle of length smaller than $g$ or it can be pruned with one of the pruning rules discussed in \cref{subsec:distDegM,subsec:mPlacement}.
    \item We use a fail-first heuristic to select the next edge to add in the algorithm. Specifically, we choose edges incident to a vertex $u$ that has the fewest eligible incident edges.
    \item The method $\algoName{enoughValidEdgesLeft}(\G, \validEdges)$ returns $\True$ if each vertex of degree $d_{1}<r$ has at least $r-d_{1}$ incident edges left in $\validEdges$ and each vertex of degree $d_{2}>r$ has at least $m-d_{2}$ incident edges left in $\validEdges$. Otherwise, it returns $\False$.
    \item In line~\refAlgoLine{lst:pruningChecks} of Algorithm~\ref{alg:backtrack2} we check whether the graph obtained by adding the edge $\edge{u}{w}$ can be pruned by checking whether a) the vertices can still obtain the correct degree (using the method $\algoName{enoughValidEdgesLeft}$), b) the $\dmin$ constraint is violated (as described in \cref{subsec:distDegM}), c) the maximal $m$-placement constraint is violated (as described in \cref{subsec:mPlacement}), d) the pruning rule from \cite{Exoo2011} can be applied and e) the ``Isolated points'' pruning rule from Section 3 of~\cite{McKay1998} can be applied. The latter pruning rule works as follows: suppose $V_0 = \{v_0, \ldots, v_k\}$ is the set of isolated vertices in the current constructed graph and $u$ is an arbitrary vertex. Then adding the edge $\edge{u}{v_0}$ results in a graph which is isomorphic to the graph obtained by adding $\edge{u}{v_i}$ for $i \in \{1,\ldots,k\}$. Therefore we only allow one of these edges to be added and prune the rest. 
\end{itemize}

\begin{algorithm}[htb]
	\caption{\algoName{genBiregGirthGraphs}($n,r,m,g)$}\label{alg:genBiregGirthGraphs}
    \If{$n < M\rmg$}{
        \Comment{If the order $n$ is smaller than $M\rmg$, there are definitely no $\rmg$-graphs.}
        \Return
    }
	$\dmin \gets \algoName{minDistDegM}(n,r,m,g)$ \\
	$\cmax \gets \algoName{maxMPlacement}(n,r,m,g)$ \\
    $\T \gets \algoName{makeBMT}(n,r,m,g)$\\
    \Comment{Add $n-M\rmg$ isolated vertices to BMT $T$.}
    $\G_{start} \gets \algoName{addIsoVertices}(T, n-M\rmg)$\\
    \Comment{Determine which edges can be added without exceeding the girth $g$, maximum degree $m$ and not violating $\cmax$, $\dmin$}
    $\validEdges \gets \algoName{calcValidAddableEdges}(\G_{start}, m, g, \cmax, \dmin)$\\
    \Comment{Heuristic: edges will be added incident with the vertex $u$ that has the fewest options in $\validEdges$}
    $u \gets \arg\min_{w \in V(G), \degree{w}<r}(|\{e \colon e \in \validEdges \text{ and } w \in e \}|)$ \\
    $\algoName{\backtrack}(\G_{start}, \validEdges, \cmax, \dmin, r, m, g, u, \True)$ \Comment{See Algorithm~\ref{alg:backtrack2}}
\end{algorithm}

\begin{algorithm}[htbp]
\DontPrintSemicolon
	\caption{\algoName{\backtrack}($\G, \validEdges, \cmax, \dmin, r, m, g, u,\checkVisited$)}\label{alg:backtrack2}
	\Comment{$u$: current vertex to add edges to}
    \Comment{$\checkVisited$: $\True$ if $\G$ has not been called with $\backtrack$ before; otherwise $\False$}
	\If{$\checkVisited$}{
		\If{method was called with graph that is isomorphic with $\G$}{\Return}
        \If(\Comment*[h]{$\G$ is an $\rmg$-graph}.){$\G$ has degree set $\{r,m\}$}{
			output $\G$ \\ 
            \Comment{All vertices have degree $r$ or $m$. Try to add edges between vertices of degree $r$ until they obtain degree $m$.}
            $u \gets \arg\min_{w \in V(G), \degree{w}=r}(|\{e \colon e \in \validEdges \text{ and } w \in e \}|)$ \\
			\While{$\algoName{enoughValidEdgesLeft}(\G, \validEdges)$}{
                \Comment{In this case, we call $\backtrack$ with the same graph $\G$, so we don't want to do an isomorphism check and set $\checkVisited$ to $\False$.}
				$\algoName{\backtrack}(\G, \validEdges, \cmax, \dmin, r, m, g, u, \False)$ \\
                \Comment{Forbid adding edges to $u$.}
                $\validEdges \gets \validEdges \setminus \{e \colon e \in \validEdges \text{ and } u \in e \}$ \\
                $u \gets \arg\min_{w \in V(G), \degree{w}=r}(|\{e \colon e \in \validEdges \text{ and } w \in e \}|)$ \\
			}
			\Return \\
		}
        \Comment{$\G$ has vertices which are not of degree $r$ or $m$. Choose new $u$, if the current vertex $u$ has been completed to degree $r$ or $m$.}
    	\If{$\degree{u}=r$ and there is a vertex $v$ such that $\degree{v}<r$}{
            $u \gets \arg\min_{w \in V(G), \degree{w}<r}(|\{e \colon e \in \validEdges \text{ and } w \in e \}|)$ \\
        } 
        \ElseIf{$\degree{u}=r$ or $\degree{u}=m$}{
                $u \gets \arg\min_{w \in V(G), r<\degree{w}<m}(|\{e \colon e \in \validEdges \text{ and } w \in e \}|)$ \\
            }
    	}
    \Comment{Apply pruning rule from \cite{Exoo2011} on all $\validEdges$.}
		$\validEdges \gets \algoName{pruneValidEdgesExoo}(\G, \validEdges)$ \\
		\If{$\neg \algoName{enoughValidEdgesLeft}(\G, \validEdges)$}{\Return}
    \Comment{Try adding each valid edge incident to $u$.}    
    \ForEach{$w \in \{w \colon \edge{u}{w} \in \validEdges\})$}{
            \Comment{Copy $\validEdges$ in $prev\validEdges$ to restore it after removing $\edge{u}{w}$ again.}
            $prev\validEdges \gets \validEdges$\\
			$\E(\G) \gets \E(\G) \cup \{\edge{u}{w}\}$\\
            \Comment{Remove edges from $\validEdges$ resulting from adding edge $\edge{u}{w}$.}
            $\validEdges \gets \algoName{updateValidAddableEdges}(\G, \validEdges, \cmax, \dmin,r,m,g)$\\
            \Comment{Check if we can prune $\edge{u}{w}$.}
            \If{none of the pruning rules can be applied}{ \label{lst:pruningChecks}
				$\algoName{\backtrack}(\G, \validEdges, \cmax, \dmin, r, m, g, u, \True)$ \\
			}
			$\E(\G) \gets \E(\G) \setminus \{\edge{u}{w}\}$\\
            \Comment{Restore $\validEdges$ and remove $\edge{u}{w}$ from it.}
            $\validEdges \gets prev\validEdges \setminus \{\edge{u}{w}\}$\\
			\If{$ \neg \algoName{enoughValidEdgesLeft}(\G, \validEdges)$}{
				\Return \\
			}
	}
\end{algorithm}

\subsection{Experiments and results}\label{subsec:exh_gen_exp}

In this subsection we first present the exhaustive lists we obtained with the algorithm in \cref{sec:exp_exh_lists} and after that, we present the improved lower bounds in \cref{sec:exp_low_bounds}. We recall that our implementation of the algorithm can be downloaded from \url{https://github.com/tiboat/biregGirthGraphs} \cite{repo} and in \cref{sec:sanity} we describe the extra measures we took to ensure the correctness of our implementation.

\subsubsection{Exhaustive lists}\label{sec:exp_exh_lists}
\cref{tab:exh_lists} presents the number of pairwise non-isomorphic bi-regular cages $\size{\rmgnFam{r}{m}{g}{n\rmg}}$ found with the exhaustive generation algorithm -- which we will denote as $\biregGen$ -- for $\numExhLists$ triples of $(r,m,g)$. The obtained graphs are also available at \cite{repo}. This result illustrates one of the strengths of our algorithm, since exhaustive lists of $\rmg$-cages are only known in few cases. 
\cref{tab:exh_lists} also presents some additional information: the order $n\rmg$, the actual minimum distance $\dmin\left(\rmgnFam{r}{m}{g}{n\rmg}\right)$ between vertices of degree $m$ in graphs in $\rmgnFam{r}{m}{g}{n\rmg}$, i.e.
\begin{equation*}
    \dmin\left(\rmgnFam{r}{m}{g}{n\rmg}\right) = \min\limits_{\G \in \rmgnFam{r}{m}{g}{n\rmg}}  \min\limits_{v_1,v_2 \in \V_m,  v_1 \neq v_2} \distwG{\G}{v_1}{v_2},
\end{equation*}
and the minimum and maximum amount of vertices of degree $m$ in graphs in $\rmgnFam{r}{m}{g}{n\rmg}$. Note that in case all graphs in $\rmgnFam{r}{m}{g}{n\rmg}$ only have one vertex of degree $m$, then $\dmin\left(\rmgnFam{r}{m}{g}{n\rmg}\right)= \infty$. The execution time ranges from $0.004$ seconds for $\rmgIn{3}{4}{5}$ to roughly 17 days for $\rmgIn{3}{9}{6}$. In the latter case we split up the computation among multiple processes.
A selection of the bi-regular cages from \cref{tab:exh_lists} can also be inspected at \textit{the House of Graphs} \cite{HOG} by searching for the keyword ``bi-regular cage''.

We now discuss some observations following from the data of \cref{tab:exh_lists}. Firstly, it holds that $\size{\rmgnFam{r}{m}{g}{n\rmg}} < \size{\rmgnFam{r}{m+1}{g}{n\rmgIn{r}{m+1}{g}}}$, for all $(r,m,g)$ in the table. Secondly, while for the classical Cage Problem there are only a few cages in all cases for which the exhaustive lists are known, the situation seems very different for bi-regular cages. For example, we determined that there are exactly 154 290 pairwise non-isomorphic $\rmgIn{3}{9}{6}$-cages. Moreover, $\dmin\left(\rmgnFam{r}{m}{g}{n\rmg}\right)$ is at least equal to $2$ for each $(r,m,g)$ in the table, which lists graphs of girth $g \geq 5$. Yet, for $g=3$ vertices of degree $m$ can be adjacent. Nonetheless, we can still ask the question if vertices of degree $m$ are non-adjacent for $g\geq4$.
For $g=4$, there can be no adjacent vertices of degree $m$, since when placing two adjacent vertices of degree $m$ in the root of the BMT, the order of the BMT, and thus a lower bound, is equal to $2m$, which is strictly larger than $n\rmgIn{r}{m}{4}=r+m$. Hence, we can formulate our problem as follows.

\begin{p}
    Does there exist an $\rmg$-graph with $g\geq5$ in which two vertices of degree $m$ are adjacent?
\end{p}

Thirdly, we observe that the amount of vertices of degree $m$ is relatively small in comparison to the amount of vertices of degree $r$. More precisely, the largest ratio $n\rmg / \size{V_m}$ in the table is $0.231$ for $\rmgIn{3}{4}{5}$. The latter observation is intuitively sensible, since fewer vertices of degree $m$ lead to fewer edges, which makes it easier to avoid creating cycles shorter than the desired girth. Fourthly, we note that $n\rmgIn{3}{4}{10}=82$ is a rather recent result  \cite{biregEven} where they also present a $\rmgIn{3}{4}{10}$-cage. With $\biregGen$ we were able to show that this $\rmgIn{3}{4}{10}$-cage is unique. 

Lastly, we remark that some of the graphs that we found have interesting properties. For example, Fig.~\ref{fig:375} shows a $(\{3,7\};5)$-cage, which is \textit{hypohamiltonian} and appeared in~\cite{hypoham17} (a graph $G$ is hypohamiltonian if $G$ is not hamiltonian, but $G-v$ is hamiltonian for each $v \in V(G)$). It is a member of the following infinite family of $(\{3,m\};5)$-cages, which generalizes the Petersen graph: for each integer $m \geq 3$ the family contains one graph with vertex set $\{a_0\} \cup \{b_0,b_1,\ldots,b_{m-1}\} \cup \{c_0,c_1,\ldots,c_{2m-1}\}$ and edge set $\{\{a_0,b_i\}~|~ 0 \leq i \leq m-1\} \cup \{\{c_i,c_{i+1}\}~|~ 0 \leq i \leq 2m-1\} \cup \{\{b_i,c_{2i}\}~|~ 0 \leq i \leq m-1\} \cup \{\{b_i,c_{2i+3}\}~|~ 0 \leq i \leq m-1\}$, (where indices are taken modulo $m$ or $2m$). We remark that a different family of $(\{3,m\};5)$-cages is known (for example as described in~\cite{Dn-cages}). Fig.~\ref{fig:346And455} shows one of the $(\{3,4\};6)$-cages and one of the $(\{4,5\};5)$-cages. These graphs can be obtained by adding appropriate edges between a) the $(3,4)$ Moore tree and the $12$-cycle and b) the $(4,3)$ Moore tree and the M\"obius-Kantor graph, respectively.

        \begin{figure}[h!]
\begin{center}
\begin{tikzpicture}[scale=0.65]
        \foreach \x in {0,1,...,6}{
        \draw[fill] (\x*360/7:2.0) circle (1.8pt);
        \draw ({360/7 * (\x)}:2.0) -- (0,0);

        \draw ({360/7 * (\x)}:2.0) -- ({360/14 * (2*\x) -30}:4.0);
        \draw ({360/7 * (\x)}:2.0) -- ({360/14 * (2*\x+3)-30}:4.0);
        }
        
        \foreach \x in {0,1,...,13}{
        \draw[fill] (\x*360/14-30:4.0) circle (1.8pt);
        \draw ({360/14* (\x) -30}:4.0) -- ({360/14 * (\x+1) -30}:4.0);
        }
        \draw[fill] (0,0) circle (1.8pt);
\end{tikzpicture}
\end{center}
\caption{One of the $(\{3,7\};5)$-cages, which is also hypohamiltonian.}\label{fig:375}
\end{figure}

        \begin{figure}[h!]
\begin{center}
\begin{tikzpicture}[scale=0.65]
        \foreach \x in {0,1,...,11}{
        \draw[fill] (\x*360/12:4.0) circle (1.8pt);

        \draw ({360/12 * (\x)}:4.0) -- ({360/12 * (\x+1)}:4.0);
        }
        \foreach \x in {0,1,...,3}{
        \draw[fill] (\x*360/4+45:2.5) circle (1.8pt);
        }

        \draw (0*360/4+45:2.5) -- ({360/12 * (3*2)}:4.0);
        \draw (0*360/4+45:2.5) -- ({360/12 * (3*2+4)}:4.0);
        \draw (0*360/4+45:2.5) -- ({360/12 * (3*2+8)}:4.0);

        \draw (1*360/4+45:2.5) -- ({360/12 * (3*1)}:4.0);
        \draw (1*360/4+45:2.5) -- ({360/12 * (3*1+4)}:4.0);
        \draw (1*360/4+45:2.5) -- ({360/12 * (3*1+8)}:4.0);

        \draw (2*360/4+45:2.5) -- ({360/12 * (3*3)}:4.0);
        \draw (2*360/4+45:2.5) -- ({360/12 * (3*3+4)}:4.0);
        \draw (2*360/4+45:2.5) -- ({360/12 * (3*3+8)}:4.0);

        \draw (3*360/4+45:2.5) -- ({360/12 * (3*4)}:4.0);
        \draw (3*360/4+45:2.5) -- ({360/12 * (3*4+4)}:4.0);
        \draw (3*360/4+45:2.5) -- ({360/12 * (3*4+8)}:4.0);

        \draw (-1,0) -- (1*360/4+45:2.5);
        \draw (-1,0) -- (2*360/4+45:2.5);
        \draw (1,0) -- (0*360/4+45:2.5);
        \draw (1,0) -- (3*360/4+45:2.5);
        \draw (-1,0) -- (1,0);
        \draw[fill] (-1,0) circle (1.8pt);
        \draw[fill] (1,0) circle (1.8pt);
\end{tikzpicture}\quad
\begin{tikzpicture}[scale=0.65]
        \foreach \x in {0,1,...,15}{
        \draw[fill] (\x*360/16+34:4.0) circle (1.8pt);

        \draw ({360/16 * (\x)+34}:4.0) -- ({360/16 * (\x+1)+34}:4.0);
        }

        \foreach \x in {0,2,4,6,8,10,12,14}{
        \draw ({360/16 * (\x)+34}:4.0) -- ({360/16 * (\x+5)+34}:4.0);
        }
        \draw ({360/16 * (4)+34}:4.0) -- (0,1.5);
        \draw ({360/16 * (4+3)+34}:4.0) -- (0,1.5);
        \draw ({360/16 * (4-5)+34}:4.0) -- (0,1.5);
        \draw ({360/16 * (4-5-3)+34}:4.0) -- (0,1.5);

        \draw ({360/16 * (8)+34}:4.0) -- (-1.5,0);
        \draw ({360/16 * (8+3)+34}:4.0) -- (-1.5,0);
        \draw ({360/16 * (8-5)+34}:4.0) -- (-1.5,0);
        \draw ({360/16 * (8-5-3)+34}:4.0) -- (-1.5,0);

        \draw ({360/16 * (6)+34}:4.0) -- (0,-1.5);
        \draw ({360/16 * (6+3)+34}:4.0) -- (0,-1.5);
        \draw ({360/16 * (6-5)+34}:4.0) -- (0,-1.5);
        \draw ({360/16 * (6-5-3)+34}:4.0) -- (0,-1.5);

        \draw ({360/16 * (2)+34}:4.0) -- (1.5,0);
        \draw ({360/16 * (2+3)+34}:4.0) -- (1.5,0);
        \draw ({360/16 * (2-5)+34}:4.0) -- (1.5,0);
        \draw ({360/16 * (2-5-3)+34}:4.0) -- (1.5,0);

        \draw (0,0) -- (1.5,0);
        \draw (0,0) -- (-1.5,0);
        \draw (0,0) -- (0,1.5);
        \draw (0,0) -- (0,-1.5);
        
        \draw[fill] (0,0) circle (1.8pt);
        \draw[fill] (1.5,0) circle (1.8pt);
        \draw[fill] (-1.5,0) circle (1.8pt);
        \draw[fill] (0,1.5) circle (1.8pt);
        \draw[fill] (0,-1.5) circle (1.8pt);
\end{tikzpicture}
\end{center}
\caption{One of the $(\{3,4\};6)$-cages (left) and one of the $(\{4,5\};5)$-cages (right).}\label{fig:346And455}
\end{figure}

\begin{table}[htbp]
	\centering
	\begin{tabular}{r r r r r r r  r r }
		\toprule                
		& & & & & & \multicolumn{2}{c}{$\size{\V_m}$}  \\ 
		\cmidrule(lr){7-8}
		$r$ & $m$ & $g$ & $n\rmg$ & $\size{\rmgnFam{r}{m}{g}{n\rmg}}$ & $\dmin\left(\rmgnFam{r}{m}{g}{n\rmg}\right)$ & \multicolumn{1}{c}{min} & \multicolumn{1}{c}{max} \\ 
		\midrule
        3 & 4 & 5 & 13 & 4 & 2 & 1 & 3 \\
        3 & 4 & 6 & 18 & 2 & 2 & 2 & 4 \\
        3 & 4 & 7 & 29 & 11 & $\infty$ & 1 & 1 \\
        3 & 4 & 8 & 39 & 2 & 4 & 3 & 3 \\
        3 & 4 & 9 & 61 & 15 & $\infty$ & 1 & 1 \\
        3 & 4 & 10 & 82 & 1 & 4 & 4 & 4 \\
        3 & 5 & 5 & 16 & 20 & 2 & 1 & 3 \\
        3 & 5 & 6 & 22 & 6 & 2 & 2 & 4 \\
        3 & 5 & 7 & 36 & 32 648 & 3 & 1 & 2 \\
        3 & 5 & 8 & 48 & 44 & 4 & 3 & 3 \\
        3 & 6 & 5 & 19 & 92 & 2 & 1 & 3 \\
        3 & 6 & 6 & 26 & 18 & 2 & 2 & 4 \\
        3 & 6 & 8 & 55 & 6 & 4 & 5 & 5 \\
        3 & 7 & 5 & 22 & 792 & 2 & 1 & 3 \\
        3 & 7 & 6 & 30 & 60 & 2 & 2 & 4 \\
        3 & 8 & 5 & 25 & 9 782 & 2 & 1 & 3 \\
        3 & 8 & 6 & 34 & 322 & 2 & 2 & 4 \\
        3 & 9 & 5 & 28 & 154 290 & 2 & 1 & 3 \\
        3 & 9 & 6 & 38 & 2 239 & 2 & 2 & 4 \\
        4 & 5 & 5 & 21 & 3 & 2 & 2 & 4 \\
        4 & 5 & 6 & 32 & 13 & 2 & 2 & 6 \\
        4 & 6 & 5 & 25 & 34 858 & 2 & 1 & 5 \\
        5 & 6 & 5 & 31 & 6 & 2 & 1 & 5 \\
        6 & 7 & 5 & 43 & 1 & 2 & 6 & 6 \\
        \bottomrule
	\end{tabular}
	\caption{The number of pairwise non-isomorphic $\rmg$-cages $\size{\rmgnFam{r}{m}{g}{n\rmg}}$ generated by $\biregGen$ together with the order $n\rmg$, the minimum distance $\dmin$ between vertices of degree $m$, and the minimum and maximum number of vertices of degree $m$.}
	\label{tab:exh_lists}
\end{table}

\subsubsection{Improved lower bounds}\label{sec:exp_low_bounds}
The exhaustive generation algorithm allows us to improve lower bounds as well. That is because if the algorithm does not produce any graphs for the order $n$ equal to the lower bound $n_{\textrm{low}}$, then we know that the lower bound is at least $n_{\textrm{low}}+1$. We were able to improve the lower bound of $n\rmgIn{4}{5}{7}$ from 66 to 69, by executing the algorithm for orders 66, 67 and 68 and thus obtaining no graphs in the output for each of these orders. The total running time of these computations was roughly 189 CPU days of which the computation for order 68 took 185 CPU days.

\section{Constructions for bi-regular graphs out of regular graphs}\label{sec:constr}

In this section we will present three straightforward constructions that construct bi-regular graphs with degree set $\{r,m\}$ based on $r$-regular graphs. The three constructions require a different condition on $m$.

\begin{constr}[$r$-regular to degree set $\{r,r+1\}$]\label{constr:addEdge}
    Given an $r$-regular graph $\Greg$, add an edge between two different vertices of $\Greg$.
\end{constr}

When trying to construct an $\rmg$-graph $\Gbireg$ with \cref{constr:addEdge} we can limit ourselves to $r$-regular graphs with $g(\Greg) \geq g$ that have two vertices at distance at least $g-1$. That is because, firstly, adding an edge cannot increase the girth. Secondly, the two vertices $v,v'$ that will be connected with an edge need to be at distance at least $g-1$ in $\Greg$, since otherwise this would lead to a cycle of length smaller than $g$ and thus $g(\Gbireg)<g$.

\begin{constr}[$r$-regular to degree set $\{r,2t\}$]\label{constr:even}
Given an $r$-regular graph $\Greg$ with $r \geq 2$. The following steps lead to a graph $\Gbireg$ with degree set $\{r,m=2t\}$ where $t \geq 1$.
\begin{enumerate}[noitemsep]
    \item Take a set of $t$ non-adjacent edges $\Edel \subset \E(\Greg)$ and delete these edges from the graph.
	\item Add a vertex $w_1$.
	\item Add an edge between $w_1$ and every vertex incident to an edge of $\Edel$.
\end{enumerate}
\end{constr}

The constructed graph $\Gbireg$ in \cref{constr:even} has degree set $\{r,m\}$ since the new vertex $w_1$ is of degree $m$ and all other vertices are of degree $r$.

\begin{constr}[$r$-regular to degree set $\{r,2t+1\}$]\label{constr:odd}
Given an $r$-regular graph $\Greg$ with $r \geq 2$. The following steps lead to a graph $\Gbireg$ with degree set $\{r,m=2t+1\}$ where $t \geq 1$.
\begin{enumerate}[noitemsep]
    \item Take a set of $2t$ non-adjacent edges $\Edel \subset \E(\Greg)$ and delete these edges from the graph.
	\item Add two vertices $w_1, w_2$ and add the edge $\edge{w_1}{w_2}$.
	\item Add an edge between $w_1$ and exactly one vertex incident to each edge in $\Edel$.
    \item Add an edge between $w_2$ and the incident vertex of every edge in $\Edel$ that did not connect to $w_1$.
\end{enumerate}
\end{constr}

The constructed graph $\Gbireg$ in \cref{constr:odd} has degree set $\{r,m\}$ since the new vertices $w_1,w_2$ are of degree $m$ and all other vertices are of degree $r$. For \cref{constr:even,constr:odd} we can also restrict ourselves to $r$-regular graphs which have a given minimum girth and enough edges at a large enough distance to be able to obtain a graph of a specified girth $g$. For more information on these restrictions, we refer to \cite[Section~4.2]{thesis}.

\subsection{Experiments and results}\label{sec:constr_exp}

\newcommand{\triVT}{\textsc{3-VT}}
\newcommand{\triCay}{\textsc{3-Cay}}
\newcommand{\triAT}{\textsc{3-1T}}
\newcommand{\triSS}{\textsc{3-SS}}
\newcommand{\tetraAT}{\textsc{4-1T}}
\newcommand{\tetraTwoAT}{\textsc{4-2T}}
\newcommand{\tetraET}{\textsc{4-ET}}
\newcommand{\pentaAT}{\textsc{5-1T}}

We consulted exhaustive lists of $r$-regular graphs available in the literature as seeds for our constructions. \cref{tab:reg_lists} shows the used lists of $r$-regular graphs with the value of $r$, the maximum order in the list and which property the graphs in the list have in common. Since the amount of $r$-regular graphs becomes astronomically large for larger $n$, constructed lists of regular graphs usually limit themselves to specific types of graphs, more specifically, highly symmetric graphs.
For more information on these highly symmetric graphs, we refer to \cite{algGraphTheoryBook}. The constructions led to the improvement of $\numImprovedUpBounds$ upper bounds, which we could not further improve using the upper bounds from \cref{sec:gen_up_bound}. 
The exact improvements can be found in \cref{sec:tables}.

Furthermore, we note that for $\rmgIn{4}{11}{9}$ we did not try the construction on every graph in $\tetraTwoAT$ (cf.\ \cref{tab:reg_lists}), since the running time of the construction code was considered to be too long on some graphs (longer than 100 hours).
Finally, we mention a limitation of two of our constructions. Recall our observation from \cref{subsec:exh_gen_exp} stating that for the generated bi-regular cages, vertices of degree $m$ are never neighbors. Yet, \cref{constr:addEdge} and \cref{constr:odd} lead to neighboring vertices of degree $m$, making it unlikely that the found $\rmg$-graphs would be bi-regular cages.

\begin{table}[htb]
	\centering
	\begin{tabular}{r r r r r r}
		\toprule
		$r$ & Max order & Source list & Paper(s) & Property & Abbreviation\\
		\midrule
		3 & 1 280 &~\cite{cubicVertTransUpTo1280:data} &~\cite{cubicVertTransUpTo1280,cubicVertTransUpTo1280ExtraTheory} & vertex transitive & $\triVT$ \\
        3 & 4 094 &~\cite{cayleyCubicUpTo4094:data} &~\cite{cubicVertTransUpTo1280} & Cayley & $\triCay$\\
        3 & 10 000 &~\cite{symmCubicUpTo10000:data} &~\cite{cubicSymm768} & 1-arc transitive & $\triAT$\\
        3 & 10 000 &~\cite{semiSymmCubicUpTo10000:data} &~\cite{cubicSymm768} & semi-symmetric & $\triSS$\\    
		4 & 512 &~\cite{edgeTransQuadric:data} &~\cite{edgeTransQuadric} & edge transitive & $\tetraET$ \\
		4 & 640 &~\cite{arcTransQuadric:data} &~\cite{cubicVertTransUpTo1280,cubicVertTransUpTo1280ExtraTheory} & 1-arc transitive & $\tetraAT$ \\
		4 & 2 000 &~\cite{2-arcTransQuadric:data} &~\cite{2-arcTransQuadric} & 2-arc transitive & $\tetraTwoAT$\\
		5 & 500 &~\cite{pentaLink} & & 1-arc transitive & $\pentaAT$ \\
		\bottomrule
	\end{tabular}
	\caption{Lists of regular graphs with degree of the vertices $r$, maximum order of a graph in the list, source of the list, possibly related papers, properties of the graphs in the list and an abbreviation for the list.}
	\label{tab:reg_lists}
\end{table}

\section[Generalized upper bound by gluing (r,g)-graphs]{Generalized upper bound by gluing $\rg$-graphs}\label{sec:gen_up_bound}

In~\cite{semiCubic2023} Aguilar, Araujo-Pardo, and Berman generalized Theorem 3 from~\cite{ConstructionsBireg}, proving the following result.

\begin{thm}[{\cite[Theorem 1]{semiCubic2023}}]\label{thm:upperBoundSemiReg}
	Let $G$ be an $(r,g)$-graph of even girth and order $n_g$ with at least two vertices at distance $g/2$. If $m=3k+t$, then there exists an $\rmg$-graph of order
	\[k(n_g-2) + 
	\begin{cases}
		2 & \text{ if } t=0,\\
		n_{g} + 2 &\text{ if } t=1,\\
		n_{g} &\text{ if } t=2.
	\end{cases}
	\]
\end{thm}

Note that when $g$ is even, one can always find at least two vertices which are at distance $g/2$ in any $\rg$-graph. The idea of the proof is to ``glue'' $(3,g)$-graphs together at these two vertices at distance $g/2$, also called \textit{remote vertices}. Based on similar ideas as presented in their proof, we generalize their theorem a) for all girths $g \geq 5$, b) for more pairs $(r,m)$ and c) by improving their bound if a stricter condition is met.

\begin{thm}\label{thm:genUpperBound}
    If there exists an $\rg$-graph of order $n_g$ with $g\geq5$ and $s$ vertices at pairwise distance at least $\ceil{g/2}$, then there exists an $\rmg$-graph of order
	
    \[k(n_g-s) + 
	\begin{cases}
		s & \text{if } m=rk,\\
		n_{g} +s&\text{if } m = rk+1  \text{ and }  s  \text{  is even},\\
		n_{g} + (l-1) (n_g - s) & \parbox{22em}{if $m = rk + l(r-1)$ for $l \geq 1$, $s$ is even and there exist $s/2$ edges at pairwise distance at least $\ceil{g/2}$.}
	\end{cases}
	\]
\end{thm}

\begin{proof}
	We divide the proof into three cases.
	\begin{itemize}
		\item  Case 1: $m=rk$. Let  $G_{1}$ and $G_{2}$ be two copies of an $(r,g)$-graph of order $n_{g}$ having $s$ vertices $v_{i,1}, v_{i,2}, \ldots, v_{i,s}$ at pairwise distance at least $\ceil{g/2}$ for $i \in \{1,2\}$, respectively.
        Construct a graph $\Gbireg$ by taking $G_1$ and $G_2$ and identifying  for each $j \in \{1, \ldots, s\}$ the vertex $v_{1,j} \in \V(\G_1)$ with $v_{2,j} \in \V(\G_2)$ and call this new vertex $x_j$. Let $X = \{x_1, \ldots x_s\}$. Suppose $\cyc$ is a cycle in $\Gbireg$. We will now show that $\cyc$ has length at least $g$. If $\cyc$ passes through at most one vertex of $X$, then this cycle is either completely contained in $G_1$ or $G_2$, because $X$ is a cut-set of $\Gbireg$, and therefore has length at least $g$. If $\cyc$ passes through at least two vertices of $X$, then the cycle consists of at least two paths of length at least $\ceil{g/2}$, and hence $\cyc$ has length at least $g$.
		
		Now let $\Gbireg$ be a graph formed by identifying $k$ copies, where the $i$-th copy is labelled $G_{i}$, at remote vertices $v_{i,1}, \ldots , v_{i,s}$ in $G_{i}$, calling the identified vertices $x_1, x_2, \ldots x_s$ as before. Since each of the graphs $G_{i}$ is $r$-regular, the identified vertices $x_1, x_2, \ldots x_s$ have degree $m = rk$. Applying the same shortest cycle analysis as above to each pair $(G_{i}, G_{j})$, it follows that the girth of $\Gbireg$ is also at least $g$,
		and $\Gbireg$ has order $k n_g-s(k-1)=k(n_g-s)+s$, with $s$ vertices of degree $m=rk$ and $k(n_{g} - s)$ vertices of degree $r$.

        \item  Case 2: $m=rk+1$ and $s$ is even. Take $k+1$ copies of an $(r,g)$-graph of order $n_{g}$ with $s$ vertices at pairwise distance at least $\ceil{g/2}$, and label the $i$-th copy as $G_{i}$. Consider $s/2$ edges $\edge{v_{1}}{w_{1}}, \ldots, \edge{v_{s/2}}{w_{s/2}}$ at pairwise distance at least $\ceil{g/2}-2$ in $G_{1}$ and delete these. These edges exist; consider, for example, an incident edge at $s/2$ of the remote vertices. Since the remote vertices have a pairwise distance of at least $\ceil{g/2}$, all these edges have a pairwise distance of at least $\ceil{g/2}-2$. Now, add $s$ vertices $v_1',w_1',\ldots,v_{s/2}',w_{s/2}'$ to $G_{1}$ such that $v_j'$ becomes a neighbor of $v_j$, and that $w_j'$ becomes a neighbor of $w_j$, for $j \in \{1,\ldots,s/2\}$. Note that the distance between any two of these $s$ vertices is at least $\ceil{g/2}$.
        
		Construct a graph $\Gbireg$ by identifying each pair $v_j', w_j'$ with two remote vertices $x_{i,j}$ and $y_{i,j}$ in each of the $k$ remaining graphs $G_{i}$. The graph $\Gbireg$ only has cycles of length $g$ or larger due to the distance requirements between the deleted edges and identified vertices. $\Gbireg$ has girth $g$ (and not larger) since it contains the cycles of length $g$ from $\G_i$ with $i > 1$. Furthermore, $\Gbireg$ has $s$ vertices of degree $rk+1$ and $k(n_{g} - s)+n_{g}$ vertices of degree $r$ for a total order of 
		$k(n_g-s)+n_g+s$.
		
		\item Case 3: $m=rk+l(r-1), l \geq 1$, $s$ is even and there exist $s/2$ edges at pairwise distance at least $\ceil{g/2}$.
        Take $k$ copies of an $(r,g)$-graph of order $n_{g}$ having $s$ vertices at pairwise distance at least $\ceil{g/2}$ and having $s/2$ edges at pairwise distance at least $\ceil{g/2}$ for $i \in \{1,\ldots,k\}$ and label the $i$-th copy as $G_{i}$. Additionally, take $l$ copies of this $(r,g)$-graph and label the $i$-th copy as $H_{i}$. In each $H_{i}$, let $\edge{v_{i,1}}{w_{i,1}}, \ldots, \edge{v_{i,s/2}}{w_{i,s/2}}$ be the $s/2$ edges at pairwise distance at least $\ceil{g/2}$. Delete these edges, and call each new graph $H'_{i}$.  Notice that all the vertices in each $H'_{i}$ have degree $r$ except the vertices incident to the deleted edges, which have degree $r-1$. Since $H_{i}$ has girth $g$, each pair of vertices $(v_{i,1},w_{i,1}), \ldots, (v_{i,s/2},w_{i,s/2})$ is now at distance at least $g-1$ from each other.

        Now, suppose that $u_{i,1}, \ldots, u_{i,s}$ are the $s$ vertices at pairwise distance at least $\ceil{g/2}$ in $G_{i}$. Construct a new graph $\Gbireg$ by identifying for $j\in \{1,\ldots,s/2\}$ all the vertices $v_{1,j}, \ldots, v_{l,j}, u_{1,j}, \ldots, u_{k,j}$, calling the new vertex $x_{j}$ and identifying all $w_{1,j}, \ldots, w_{l,j}, u_{1,s/2+j}, \ldots, u_{k,s/2+j}$, calling the new vertex $y_{j}$. As in the previous cases, we obtain a graph of girth $g$, but in this case, the $n_{g} - s$ vertices in each copy other than $x_j$ and $y_j$ have degree $r$, and $x_j$ and $y_j$ have degree $m=rk+l(r-1)$. It follows that $\Gbireg$ is an $\rmg$-graph of order $(n_{g}-s)(k+l)+s =  k(n_{g}-s) + n_g + (l-1) (n_g - s) $.

	\end{itemize}
\end{proof}

\subsection{Experiments and results}
\cref{thm:genUpperBound} depends on the number of vertices $s$ at pairwise distance $\ceil{g/2}$ from each other and also on the existence of $s/2$ edges at the same pairwise distance in case $m=rk+l(r-1)$ for $l\geq1$. Therefore we gathered $\rg$-cages or \textit{record holders} -- which are the smallest known $\rg$-graphs for specific $\rg$ -- and computed the maximum amount of vertices these graphs have at pairwise distance $\ceil{g/2}$ and the maximum amount of edges at the same pairwise distance. We limited ourselves to cages and record holders for $r\in \{3,4,5\}$ with $g\leq20$ for $r=3$ and $g\leq12$ for $r=4,5$ to stay in the same ranges of $\rg$ as the improvements of the constructions from \cref{sec:constr}. The used $\rg$-graphs are available (in \texttt{sparse6} format) at \cite{repo}. Using these $\rg$-graphs we were able to improve $\numImprovedUpBoundsThm$ upper bounds, besides the $\numImprovedUpBounds$ we improved in \cref{sec:constr}. These improvements can be found in \cref{sec:tables}. 

Lastly we note that our implementation for exactly computing the maximum amount of vertices and edges at minimum pairwise distance $\dmin$ was deemed too expensive for graphs of larger orders. Consequently, we implemented two approximate algorithms which yield a large, but not necessarily maximal, amount of vertices and edges at the minimum pairwise distance. For graph of order smaller than $728$, we used the exact algorithm and for graphs of order $728$ and larger, we ran the two approximate algorithms and selected the largest number that came out of these two.

\section{Conclusion}\label{sec:conc}

In this paper we presented computational methods for finding $\rmg$-graphs, which led to exhaustive lists of $\rmg$-cages for $\numExhLists$ triples of $(r,m,g)$, an improvement of the lower bound of $n\rmgIn{4}{5}{7}$ from 66 to 69 and the improvement of $\numImprovedUpBoundsTotal$ upper bounds. In this conclusion we present some possibilities for further work. First, one could analyze the exhaustive lists further to gain more insight into $\rmg$-graphs, and one can try to describe patterns which could potentially lead to the discovery of new infinite families of small bi-regular graphs with given girth. We again draw attention to the fact that typically there are quite a lot of bi-regular cages, which might suggest the existence of a modification step that one can apply to a bi-regular cage to obtain a different bi-regular cage. If so, it could be interesting to analyze the consequences of this modification step for the original Cage Problem as well.

Further, we note that the lower and upper bound for $n\rmgIn{3}{8}{8}$ are 74 and 75, respectively. Executing the exhaustive generation algorithm for order $74$ would thus conclude the exact order $n\rmgIn{3}{8}{8}$, making it a particularly interesting case. Yet, the computation was deemed too expensive with our algorithm. 

Moreover, the exhaustive generation algorithm might be an interesting starting point to design an algorithm to generate graphs of girth $g$ with an arbitrarily large degree set, i.e. $(\{r,m_1,\ldots,m_k\};g)$-graphs.
Furthermore, adapting the generation algorithm to planar or bipartite $\rmg$-graphs and \textit{girth-biregular} or \textit{girth-regular} graphs, as respectively discussed in \cite{biregPlanarCages}, \cite{OnBiregBipOfSmallExcess}, \cite{girthBiregular} and \cite{girth-regular} might be interesting as well.

\section*{Acknowledgements}

We thank Primož Potočnik for the support with the lists of regular graphs, which we used as seeds for our constructions from \cref{sec:constr}. In particular we thank him for making the list of 5-regular graphs available. We also thank Gunnar Brinkmann for providing us with an algorithm that we used as a sanity check for verifying the correctness of our implementation of the generation algorithm discussed in the current paper.

The computational resources and services used in this work were provided by the VSC (Flemish Supercomputer Center), funded by the Research Foundation - Flanders (FWO) and the Flemish Government – department EWI.

Jan Goedgebeur and Tibo Van den Eede are supported by an FWO grant with grant number G0AGX24N and by Internal Funds of KU Leuven. Jorik Jooken is supported by an FWO grant with grant number 1222524N.

\bibliographystyle{abbrv}
\typeout{}
\bibliography{references}

\newpage

\appendix

\section{Sanity checks implementation of exhaustive generation algorithm}\label{sec:sanity}
In this section we explain which extra steps we took to ensure correctness of the implementation of the exhaustive generation algorithm as described in \cref{sec:biregBacktrackAlgo}. This implementation is available at \url{https://github.com/tiboat/biregGirthGraphs} \cite{repo}. First, we implemented the backtracking algorithm for $\rg$-graphs from~\cite{McKay1998} with the pruning rule from~\cite{Exoo2011} and checked the correctness for the pairs $\rg$ as mentioned in \cref{tab:checks_cage_gen}. The amount of resulting graphs in the output was compared with the data from~\cite{regGraphs:data}, generated with \texttt{GENREG}~\cite{Meringer1999}, which is a generator for regular graphs. In each case we obtained the same amount of graphs.

Starting from the algorithm for $\rg$-graphs we made changes to obtain the one for $\rmg$-graphs. The output, for smaller orders, was compared with two more general generators: $\geng$ from $\nauty$~\cite{nauty} and $\multigraph$, an unpublished generator written by Gunnar Brinkmann. The latter is a generalization of the generator for cubic graphs described in~\cite{fastGenCubic}. \cref{tab:geng_mg_bg2} shows the cases we tested for $r,m,g$ and order $n$ along with the number of pairwise non-isomorphic $\rmg$-graphs $\size{\rmgnFam{r}{m}{g}{n}}$. When $\geng$ required more than 20 hours for generation, the execution was terminated. This occurred for $\rmgIn{3}{4}{7}$ with $n=29$, $\rmgIn{3}{5}{6}$ with $n=24$ and $\rmgIn{3}{6}{5}$ with $n=21$. We note that we successfully completed the generation for these cases with $\multigraph$. 
In each case mentioned in \cref{tab:geng_mg_bg2} the number of pairwise non-isomorphic $\rmg$-graphs obtained by $\geng$ or $\multigraph$ was in complete agreement with the output of our generator. 
Finally, we stress that the specialized generation algorithm discussed in the current paper was indeed worth developing, since many cases that we solved in the current paper could not be solved by $\geng$ or $\multigraph$.

\begin{table}[htbp]
    \centering
    \begin{minipage}{0.48\textwidth}
        \centering
        \begin{tabular}{r r r r }
            \toprule
            $r$ & $g$ & $n$ & number of non-isomorphic $\rg$-graphs \\
            \midrule
            3 & 5 & 10 & 1 \\
            3 & 5 & 12 & 2 \\
            3 & 5 & 14 & 9 \\
            3 & 5 & 16 & 49 \\
            3 & 6 & 14 & 1 \\
            3 & 6 & 16 & 1 \\
            3 & 6 & 18 & 5 \\
            3 & 6 & 20 & 32 \\
            3 & 7 & 24 & 1 \\
            3 & 7 & 26 & 3 \\
            3 & 8 & 30 & 1 \\
            3 & 8 & 32 & 0 \\
            4 & 5 & 19 & 1 \\
            \bottomrule
        \end{tabular}
        \caption{The number of non-isomorphic $\rg$-graphs of order $n$.}
        \label{tab:checks_cage_gen}
    \end{minipage}%
    \hfill
    \begin{minipage}{0.48\textwidth}
        \centering
        \begin{tabular}{r r r r r}
            \toprule
            $r$ & $m$ & $g$ & $n$ & $\size{\rmgnFam{r}{m}{g}{n}}$ \\
            \midrule
            3 & 4 & 5 & 13 & 4 \\
            3 & 4 & 5 & 14 & 14 \\
            3 & 4 & 5 & 15 & 149 \\
            3 & 4 & 5 & 16 & 1 498 \\
            3 & 4 & 5 & 17 & 20 398 \\
            3 & 4 & 6 & 18 & 2 \\
            3 & 4 & 6 & 19 & 6 \\
            3 & 4 & 6 & 20 & 55 \\
            3 & 4 & 6 & 21 & 304 \\
            3 & 4 & 6 & 22 & 3 269 \\
            3 & 4 & 7 & 29 & 11 \\
            3 & 5 & 5 & 16 & 20 \\
            3 & 5 & 5 & 18 & 1 139 \\
            3 & 5 & 5 & 20 & 106 562 \\
            3 & 5 & 6 & 22 & 6 \\
            3 & 5 & 6 & 24 & 753 \\
            3 & 6 & 5 & 19 & 92 \\
            3 & 6 & 5 & 20 & 99 \\
            3 & 6 & 5 & 21 & 10 568 \\
            4 & 5 & 5 & 21 & 3 \\
            4 & 5 & 5 & 22 & 51 \\
            \bottomrule
        \end{tabular}
        \caption{The number of non-isomorphic $\rmg$-graphs of order $n$.}
        \label{tab:geng_mg_bg2}
    \end{minipage}
\end{table}

\newpage

\section[Lower and upper bounds n({r,m};g)]{Lower and upper bounds $n\rmg$}\label{sec:tables}

The following tables show the best known lower bounds and  upper bounds for $n\rmg$, resulting from a literature study. 
More precisely, the bounds are shown for $r \in \{3,4,5\}$, $m \in \{r+1, \ldots, 18\}$ for $r=3$, $m \in \{r+1, \ldots, 13\}$ for $r \in \{4,5\}$, $g \in \{5,\ldots,20\}$ for $r=3$ and $g \in \{5,\ldots,12\}$ for $r \in \{4,5\}$.
For $\rmg$-cages where the exact order is known, the row is marked in gray. An upper bound is underlined 
in case this one was obtained by applying Theorem~2.2 from \cite{monotoDn-cages}, which states that $n\rmg$ increases monotonically with $g$ (and if $r$ and $m$ are odd by additionally applying the handshaking lemma in case the upper bound is odd). Furthermore, an improvement in lower or upper bound is denoted by $x \rightarrow y$, where $x$ is the lower or upper bound derived from the literature and $y$ the improved lower or upper bound obtained in this paper. Additionally, in case of an improved upper bound the column ``Source improvement'' states which theorem or construction was applied to obtain this improved bound. In case of using a construction, the abbreviation of the list of $r$-regular graphs, as introduced in \cref{tab:reg_lists}, is reported. Lastly, we clarify an upper bound from \cite{Chartrand1981}, as shown for e.g.\ $\{r,m\}=\{4,11\}$. In \cite{Chartrand1981} the existence of $\rmg$-graphs was proven by taking the union of an $\rg$-graph and an $(m,g)$-graph. Hence we can derive the following upper bound: $n\rmg \leq n\rg + n(m,g)$.

\begin{table}[h]
\centering
\scalebox{0.95}{
\begin{tabular}{r r r | r l | r l l}
\toprule
$r$ & $m$ & $g$ & $n\rmg \geq$ & Source & $n\rmg \leq$ & Source & Source improvement \\ 
\midrule
\rowcolor{gray!20}
3 & 4 & 5 & 13 & Cor. 2 \cite{ConstructionsBireg} & 13 & Cor. 2 \cite{ConstructionsBireg} & \\
\rowcolor{gray!20}
3 & 4 & 6 & 18 & Th. 2 \cite{Yang2003} & 18 & Th. 2 \cite{Yang2003} & \\
\rowcolor{gray!20}
3 & 4 & 7 & 29 & Cor. 2 \cite{ConstructionsBireg} & 29 & Cor. 2 \cite{ConstructionsBireg} & \\
\rowcolor{gray!20}
3 & 4 & 8 & 39 & Constr. 3.3 \cite{biregGirth8} & 39 & Constr. 3.3 \cite{biregGirth8} & \\
\rowcolor{gray!20}
3 & 4 & 9 & 61 & Th. 1.2 \cite{biregOdd} & 61 & Th. 1.2 \cite{biregOdd} & \\
\rowcolor{gray!20}
3 & 4 & 10 & 82 & Th. 4.1 \cite{biregEven} & 82 & Th. 4.1 \cite{biregEven} & \\
\rowcolor{gray!20}
3 & 4 & 11 & 125 & Cor. 2 \cite{ConstructionsBireg} & 125 & Cor. 2 \cite{ConstructionsBireg} & \\
3 & 4 & 12 & 163 & Th. 2.3 \cite{Aurajo-Pardo2007} & 220 & Th. 2.2 \cite{biregGirth8} & \\
3 & 4 & 13 & 253 & Moore bound & 383 & Th. 1 \cite{ConstructionsBireg} & \\
3 & 4 & 14 & 328 & Th. 2.3 \cite{Aurajo-Pardo2007} & 619 $\rightarrow$ 544 & Th. 1 \cite{ConstructionsBireg} & Constr.~\ref{constr:addEdge} \triVT\\
3 & 4 & 15 & 509 & Moore bound & 959 & Th. 1 \cite{ConstructionsBireg} & \\
3 & 4 & 16 & 659 & Th. 2.3 \cite{Aurajo-Pardo2007} & 1 852 $\rightarrow$ 1 240 & Th. 2.2 \cite{biregGirth8} & Constr.~\ref{constr:addEdge} \triVT\\
3 & 4 & 17 & 1 021 & Moore bound & 61 163 $\rightarrow$ 2 880 & Th. 3.1 \cite{semiRegCages} & Constr.~\ref{constr:addEdge} \triCay\\
3 & 4 & 18 & 1 320 & Th. 2.3 \cite{Aurajo-Pardo2007} & 124 304 $\rightarrow$ 5 120 & Th. 3.1 \cite{semiRegCages} & Th.~\ref{thm:genUpperBound}\\
3 & 4 & 19 & 2 045 & Moore bound & 249 023 $\rightarrow$ \underline{8 647} & Th. 3.1 \cite{semiRegCages} & \\
3 & 4 & 20 & 2 643 & Th. 2.3 \cite{Aurajo-Pardo2007} & 504 242 $\rightarrow$ 8 648 & Th. 3.1 \cite{semiRegCages} & Constr.~\ref{constr:addEdge} \triAT\\
\rowcolor{gray!20}
3 & 5 & 5 & 16 & Th. 5 \cite{Dn-cages} & 16 & Th. 5 \cite{Dn-cages} & \\
\rowcolor{gray!20}
3 & 5 & 6 & 22 & Th. 2 \cite{Yang2003} & 22 & Th. 2 \cite{Yang2003} & \\
\rowcolor{gray!20}
3 & 5 & 7 & 36 & Th. 7 \cite{Dn-cages} & 36 & Th. 7 \cite{Dn-cages} & \\
\rowcolor{gray!20}
3 & 5 & 8 & 48 & Constr. 3.3 \cite{biregGirth8} & 48 & Constr. 3.3 \cite{biregGirth8} & \\
\rowcolor{gray!20}
3 & 5 & 9 & 76 & Th. 1.2 \cite{biregOdd} & 76 & Th. 1.2 \cite{biregOdd} & \\
3 & 5 & 10 & 98 & Th. 2.3 \cite{Aurajo-Pardo2007} & 126 & Th. 2.2 (i) \cite{Aurajo-Pardo2007} & \\
3 & 5 & 11 & 156 & Moore bound & \underline{228} & Th. 2.1 \cite{monotoDn-cages} & \\
3 & 5 & 12 & 200 & Th. 2.3 \cite{Aurajo-Pardo2007} & 230 & Th. 2.2 \cite{biregGirth8} & \\
3 & 5 & 13 & 316 & Moore bound & \underline{692} $\rightarrow$ \underline{672} & Th. 2.1 \cite{monotoDn-cages} & \\
\bottomrule
\end{tabular}
}
\caption{Lower and upper bound for $n\rmg$ with $r=3$ and $m$ from 4 till 5.}
\end{table}

\begin{table}[h]
\centering
\scalebox{0.95}{
\begin{tabular}{r r r | r l | r l l}
\toprule
$r$ & $m$ & $g$ & $n\rmg \geq$ & Source & $n\rmg \leq$ & Source & Source improvement \\ 
\midrule
3 & 5 & 14 & 402 & Th. 2.3 \cite{Aurajo-Pardo2007} & 694 $\rightarrow$ 674 & Cor. 4 \cite{semiCubic2023} & Constr.~\ref{constr:odd} \triVT\\
3 & 5 & 15 & 636 & Moore bound & \underline{1 872} $\rightarrow$ 1 538 & Th. 2.1 \cite{monotoDn-cages} & Constr.~\ref{constr:odd} \triCay\\
3 & 5 & 16 & 808 & Th. 2.3 \cite{Aurajo-Pardo2007} & 1 874 $\rightarrow$ 1 730 & Th. 2.2 \cite{biregGirth8} & Constr.~\ref{constr:odd} \triSS\\
3 & 5 & 17 & 1 276 & Moore bound & \underline{261 664} $\rightarrow$ \underline{5 184} & Th. 2.1 \cite{monotoDn-cages} & \\
3 & 5 & 18 & 1 618 & Th. 2.3 \cite{Aurajo-Pardo2007} & 261 666 $\rightarrow$ 5 186 & Th. 2.2 (i) \cite{Aurajo-Pardo2007} & Constr.~\ref{constr:odd} \triSS\\
3 & 5 & 19 & 2 556 & Moore bound & \underline{1 047 752} & Th. 2.1 \cite{monotoDn-cages} & \\
3 & 5 & 20 & 3 240 & Th. 2.3 \cite{Aurajo-Pardo2007} & 1 047 754 & Th. 2.2 \cite{biregGirth8} & \\
\rowcolor{gray!20}
3 & 6 & 5 & 19 & Th. 5 \cite{Dn-cages} & 19 & Th. 5 \cite{Dn-cages} & \\
\rowcolor{gray!20}
3 & 6 & 6 & 26 & Th. 2 \cite{Yang2003} & 26 & Th. 2 \cite{Yang2003} & \\
\rowcolor{gray!20}
3 & 6 & 7 & 43 & Th. 7 \cite{Dn-cages} & 43 & Th. 7 \cite{Dn-cages} & \\
\rowcolor{gray!20}
3 & 6 & 8 & 55 & Th. 2.3 \cite{Aurajo-Pardo2007} & 55 & Cor. 7 \cite{ConstructionsBireg} & \\
\rowcolor{gray!20}
3 & 6 & 9 & 91 & Th. 9 \cite{Dn-cages} & 91 & Th. 9 \cite{Dn-cages} & \\
3 & 6 & 10 & 115 & Th. 2.3 \cite{Aurajo-Pardo2007} & 138 $\rightarrow$ 135 & Th. 3 (i) \cite{ConstructionsBireg} & Th.~\ref{thm:genUpperBound}\\
\rowcolor{gray!20}
3 & 6 & 11 & 187 & Th. 1.2 \cite{biregOdd} & 187 & Th. 1.2 \cite{biregOdd} & \\
3 & 6 & 12 & 235 & Th. 2.3 \cite{Aurajo-Pardo2007} & 250 $\rightarrow$ 243 & Th. 3 (i) \cite{ConstructionsBireg} & Th.~\ref{thm:genUpperBound}\\
3 & 6 & 13 & 379 & Moore bound & 542 $\rightarrow$ 487 & Th. 3 (iii) \cite{ConstructionsBireg} & Constr.~\ref{constr:even} \triVT\\
3 & 6 & 14 & 475 & Th. 2.3 \cite{Aurajo-Pardo2007} & 694 & Cor. 4 \cite{semiCubic2023} & \\
3 & 6 & 15 & 763 & Moore bound & 1 238 $\rightarrow$ 1 230 & Th. 3 (iii) \cite{ConstructionsBireg} & Th.~\ref{thm:genUpperBound}\\
3 & 6 & 16 & 955 & Th. 2.3 \cite{Aurajo-Pardo2007} & 1 918 $\rightarrow$ 1 898 & Th. 3 (i) \cite{ConstructionsBireg} & Th.~\ref{thm:genUpperBound}\\
3 & 6 & 17 & 1 531 & Moore bound & 131 066 $\rightarrow$ 3 025 & Th. 3 (iii) \cite{ConstructionsBireg} & Constr.~\ref{constr:even} \triAT\\
3 & 6 & 18 & 1 915 & Th. 2.3 \cite{Aurajo-Pardo2007} & 261 726 $\rightarrow$ 5 097 & Th. 3 (i) \cite{ConstructionsBireg} & Th.~\ref{thm:genUpperBound}\\
3 & 6 & 19 & 3 067 & Moore bound & 524 282 $\rightarrow$ 8 623 & Th. 3 (iii) \cite{ConstructionsBireg} & Th.~\ref{thm:genUpperBound}\\
3 & 6 & 20 & 3 835 & Th. 2.3 \cite{Aurajo-Pardo2007} & 1 047 846 $\rightarrow$ 10 714 & Th. 3 (i) \cite{ConstructionsBireg} & Th.~\ref{thm:genUpperBound}\\
\rowcolor{gray!20}
3 & 7 & 5 & 22 & Th. 5 \cite{Dn-cages} & 22 & Th. 5 \cite{Dn-cages} & \\
\rowcolor{gray!20}
3 & 7 & 6 & 30 & Th. 1 \cite{Hanson1992} & 30 & Th. 1 \cite{Hanson1992} & \\
\rowcolor{gray!20}
3 & 7 & 7 & 50 & Th. 7 \cite{Dn-cages} & 50 & Th. 7 \cite{Dn-cages} & \\
\rowcolor{gray!20}
3 & 7 & 8 & 66 & Constr. 3.3 \cite{biregGirth8} & 66 & Constr. 3.3 \cite{biregGirth8} & \\
\rowcolor{gray!20}
3 & 7 & 9 & 106 & Th. 9 \cite{Dn-cages} & 106 & Th. 9 \cite{Dn-cages} & \\
3 & 7 & 10 & 134 & Th. 2.3 \cite{Aurajo-Pardo2007} & 182 & Th. 2.2 (i) \cite{Aurajo-Pardo2007} & \\
\rowcolor{gray!20}
3 & 7 & 11 & 218 & Th. 1.2 \cite{biregOdd} & 218 & Th. 1.2 \cite{biregOdd} & \\
3 & 7 & 12 & 272 & Th. 2.3 \cite{Aurajo-Pardo2007} & 334 & Th. 2.2 \cite{biregGirth8} & \\
3 & 7 & 13 & 442 & Moore bound & \underline{1 040} $\rightarrow$ 810 & Th. 2.1 \cite{monotoDn-cages} & Th.~\ref{thm:genUpperBound}\\
3 & 7 & 14 & 550 & Th. 2.3 \cite{Aurajo-Pardo2007} & 1 042 $\rightarrow$ 974 & Cor. 4 \cite{semiCubic2023} & Constr.~\ref{constr:odd} \triVT\\
3 & 7 & 15 & 890 & Moore bound & \underline{2 786} $\rightarrow$ 1 850 & Th. 2.1 \cite{monotoDn-cages} & Th.~\ref{thm:genUpperBound}\\
3 & 7 & 16 & 1 104 & Th. 2.3 \cite{Aurajo-Pardo2007} & 2 788 $\rightarrow$ 2 690 & Th. 2.2 \cite{biregGirth8} & Constr.~\ref{constr:odd} \triSS\\
3 & 7 & 17 & 1 786 & Moore bound & \underline{392 466} $\rightarrow$ \underline{6 048} & Th. 2.1 \cite{monotoDn-cages} & \\
3 & 7 & 18 & 2 214 & Th. 2.3 \cite{Aurajo-Pardo2007} & 392 468 $\rightarrow$ 6 050 & Th. 2.2 (i) \cite{Aurajo-Pardo2007} & Constr.~\ref{constr:odd} \triAT\\
3 & 7 & 19 & 3 578 & Moore bound & \underline{1 571 582} $\rightarrow$ 12 948 & Th. 2.1 \cite{monotoDn-cages} & Th.~\ref{thm:genUpperBound}\\
3 & 7 & 20 & 4 432 & Th. 2.3 \cite{Aurajo-Pardo2007} & 1 571 584 $\rightarrow$ 16 090 & Th. 2.2 \cite{biregGirth8} & Th.~\ref{thm:genUpperBound}\\
\rowcolor{gray!20}
3 & 8 & 5 & 25 & Cor. 2 \cite{ConstructionsBireg} & 25 & Cor. 2 \cite{ConstructionsBireg} & \\
\rowcolor{gray!20}
3 & 8 & 6 & 34 & Th. 2 \cite{Yang2003} & 34 & Th. 2 \cite{Yang2003} & \\
\rowcolor{gray!20}
3 & 8 & 7 & 57 & Cor. 2 \cite{ConstructionsBireg} & 57 & Cor. 2 \cite{ConstructionsBireg} & \\
3 & 8 & 8 & 74 & Th. 3.3 \cite{biregGirth8} & 75 & Th. 3.4 \cite{biregGirth8} & \\
\rowcolor{gray!20}
3 & 8 & 9 & 121 & Th. 9 \cite{Dn-cages} & 121 & Th. 9 \cite{Dn-cages} & \\
\bottomrule
\end{tabular}
}
\caption{Lower and upper bound for $n\rmg$ with $r=3$ and $m$ from 5 till 8.}
\end{table}

\begin{table}[h]
\centering
\scalebox{0.95}{
\begin{tabular}{r r r | r l | r l l}
\toprule
$r$ & $m$ & $g$ & $n\rmg \geq$ & Source & $n\rmg \leq$ & Source & Source improvement \\ 
\midrule
3 & 8 & 10 & 151 & Th. 2.3 \cite{Aurajo-Pardo2007} & 206 & Cor. 2 \cite{semiCubic2023} & \\
\rowcolor{gray!20}
3 & 8 & 11 & 249 & Cor. 2 \cite{ConstructionsBireg} & 249 & Cor. 2 \cite{ConstructionsBireg} & \\
3 & 8 & 12 & 308 & Th. 2.3 \cite{Aurajo-Pardo2007} & 374 & Cor. 3 \cite{semiCubic2023} & \\
3 & 8 & 13 & 505 & Moore bound & 765 $\rightarrow$ 673 & Th. 1 \cite{ConstructionsBireg} & Constr.~\ref{constr:even} \triVT\\
3 & 8 & 14 & 623 & Th. 2.3 \cite{Aurajo-Pardo2007} & 1 040 & Cor. 4 \cite{semiCubic2023} & \\
3 & 8 & 15 & 1 017 & Moore bound & 1 917 $\rightarrow$ 1 729 & Th. 1 \cite{ConstructionsBireg} & Constr.~\ref{constr:even} \triSS\\
3 & 8 & 16 & 1 252 & Th. 2.3 \cite{Aurajo-Pardo2007} & 2 876 & Th. 1 \cite{semiCubic2023} & \\
3 & 8 & 17 & 2 041 & Moore bound & 261 725 $\rightarrow$ 5 185 & Th. 1 \cite{ConstructionsBireg} & Constr.~\ref{constr:even} \triSS\\
3 & 8 & 18 & 2 511 & Th. 2.3 \cite{Aurajo-Pardo2007} & 392 588 $\rightarrow$ 8 065 & Th. 1 \cite{semiCubic2023} & Constr.~\ref{constr:even} \triAT\\
3 & 8 & 19 & 4 089 & Moore bound & 1 047 845 & Th. 1 \cite{ConstructionsBireg} & \\
3 & 8 & 20 & 5 028 & Th. 2.3 \cite{Aurajo-Pardo2007} & 1 571 768 & Th. 1 \cite{semiCubic2023} & \\
\rowcolor{gray!20}
3 & 9 & 5 & 28 & Th. 5 \cite{Dn-cages} & 28 & Th. 5 \cite{Dn-cages} & \\
\rowcolor{gray!20}
3 & 9 & 6 & 38 & Th. 2 \cite{Yang2003} & 38 & Th. 2 \cite{Yang2003} & \\
\rowcolor{gray!20}
3 & 9 & 7 & 64 & Th. 7 \cite{Dn-cages} & 64 & Th. 7 \cite{Dn-cages} & \\
\rowcolor{gray!20}
3 & 9 & 8 & 80 & Th. 2.3 \cite{Aurajo-Pardo2007} & 80 & Cor. 7 \cite{ConstructionsBireg} & \\
\rowcolor{gray!20}
3 & 9 & 9 & 136 & Th. 9 \cite{Dn-cages} & 136 & Th. 9 \cite{Dn-cages} & \\
3 & 9 & 10 & 168 & Th. 2.3 \cite{Aurajo-Pardo2007} & 206 $\rightarrow$ 200 & Th. 3 (i) \cite{ConstructionsBireg} & Th.~\ref{thm:genUpperBound}\\
\rowcolor{gray!20}
3 & 9 & 11 & 280 & Th. 1.2 \cite{biregOdd} & 280 & Th. 1.2 \cite{biregOdd} & \\
3 & 9 & 12 & 344 & Th. 2.3 \cite{Aurajo-Pardo2007} & 374 $\rightarrow$ 360 & Th. 3 (i) \cite{ConstructionsBireg} & Th.~\ref{thm:genUpperBound}\\
\rowcolor{gray!20}
3 & 9 & 13 & 568 & Th. 1.2 \cite{biregOdd} & 568 & Th. 1.2 \cite{biregOdd} & \\
3 & 9 & 14 & 696 & Th. 2.3 \cite{Aurajo-Pardo2007} & 1 040 & Cor. 4 \cite{semiCubic2023} & \\
3 & 9 & 15 & 1 144 & Moore bound & 1 856 $\rightarrow$ 1 840 & Th. 3 (iii) \cite{ConstructionsBireg} & Th.~\ref{thm:genUpperBound}\\
3 & 9 & 16 & 1 400 & Th. 2.3 \cite{Aurajo-Pardo2007} & 2 876 $\rightarrow$ 2 836 & Th. 3 (i) \cite{ConstructionsBireg} & Th.~\ref{thm:genUpperBound}\\
3 & 9 & 17 & 2 296 & Moore bound & 196 598 $\rightarrow$ 6 488 & Th. 3 (iii) \cite{ConstructionsBireg} & Th.~\ref{thm:genUpperBound}\\
3 & 9 & 18 & 2 808 & Th. 2.3 \cite{Aurajo-Pardo2007} & 392 588 $\rightarrow$ 7 634 & Th. 3 (i) \cite{ConstructionsBireg} & Th.~\ref{thm:genUpperBound}\\
3 & 9 & 19 & 4 600 & Moore bound & 786 422 $\rightarrow$ 12 922 & Th. 3 (iii) \cite{ConstructionsBireg} & Th.~\ref{thm:genUpperBound}\\
3 & 9 & 20 & 5 624 & Th. 2.3 \cite{Aurajo-Pardo2007} & 1 571 768 $\rightarrow$ 16 052 & Th. 3 (i) \cite{ConstructionsBireg} & Th.~\ref{thm:genUpperBound}\\
\rowcolor{gray!20}
3 & 10 & 5 & 31 & Th. 5 \cite{Dn-cages} & 31 & Th. 5 \cite{Dn-cages} & \\
\rowcolor{gray!20}
3 & 10 & 6 & 42 & Th. 2 \cite{Yang2003} & 42 & Th. 2 \cite{Yang2003} & \\
\rowcolor{gray!20}
3 & 10 & 7 & 71 & Th. 7 \cite{Dn-cages} & 71 & Th. 7 \cite{Dn-cages} & \\
3 & 10 & 8 & 91 & Th. 3.3 \cite{biregGirth8} & 93 & Th. 3.4 \cite{biregGirth8} & \\
\rowcolor{gray!20}
3 & 10 & 9 & 151 & Th. 9 \cite{Dn-cages} & 151 & Th. 9 \cite{Dn-cages} & \\
3 & 10 & 10 & 186 & Th. 2.3 \cite{Aurajo-Pardo2007} & 276 $\rightarrow$ 272 & Cor. 2 \cite{semiCubic2023} & Th.~\ref{thm:genUpperBound}\\
\rowcolor{gray!20}
3 & 10 & 11 & 311 & Th. 1.2 \cite{biregOdd} & 311 & Th. 1.2 \cite{biregOdd} & \\
3 & 10 & 12 & 381 & Th. 2.3 \cite{Aurajo-Pardo2007} & 500 $\rightarrow$ 488 & Cor. 3 \cite{semiCubic2023} & Th.~\ref{thm:genUpperBound}\\
\rowcolor{gray!20}
3 & 10 & 13 & 631 & Th. 1.2 \cite{biregOdd} & 631 & Th. 1.2 \cite{biregOdd} & \\
3 & 10 & 14 & 770 & Th. 2.3 \cite{Aurajo-Pardo2007} & 1 388 $\rightarrow$ 1 361 & Cor. 4 \cite{semiCubic2023} & Constr.~\ref{constr:even} \triCay\\
3 & 10 & 15 & 1 271 & Moore bound & \underline{3 835} $\rightarrow$ 2 460 & Th. 2.1 \cite{monotoDn-cages} & Th.~\ref{thm:genUpperBound}\\
3 & 10 & 16 & 1 549 & Th. 2.3 \cite{Aurajo-Pardo2007} & 3 836 $\rightarrow$ 3 796 & Th. 1 \cite{semiCubic2023} & Th.~\ref{thm:genUpperBound}\\
3 & 10 & 17 & 2 551 & Moore bound & \underline{523 451} $\rightarrow$ 6 049 & Th. 2.1 \cite{monotoDn-cages} & Constr.~\ref{constr:even} \triAT\\
3 & 10 & 18 & 3 106 & Th. 2.3 \cite{Aurajo-Pardo2007} & 523 452 $\rightarrow$ 10 196 & Th. 1 \cite{semiCubic2023} & Th.~\ref{thm:genUpperBound}\\
3 & 10 & 19 & 5 111 & Moore bound & \underline{2 095 691} $\rightarrow$ 17 248 & Th. 2.1 \cite{monotoDn-cages} & Th.~\ref{thm:genUpperBound}\\
3 & 10 & 20 & 6 221 & Th. 2.3 \cite{Aurajo-Pardo2007} & 2 095 692 $\rightarrow$ 21 428 & Th. 1 \cite{semiCubic2023} & Th.~\ref{thm:genUpperBound}\\
\rowcolor{gray!20}
3 & 11 & 5 & 34 & Th. 5 \cite{Dn-cages} & 34 & Th. 5 \cite{Dn-cages} & \\
\bottomrule
\end{tabular}
}
\caption{Lower and upper bound for $n\rmg$ with $r=3$ and $m$ from 8 till 11.}
\end{table}

\begin{table}[h]
\centering
\scalebox{0.95}{
\begin{tabular}{r r r | r l | r l l}
\toprule
$r$ & $m$ & $g$ & $n\rmg \geq$ & Source & $n\rmg \leq$ & Source & Source improvement \\ 
\midrule
\rowcolor{gray!20}
3 & 11 & 6 & 46 & Th. 1 \cite{Hanson1992} & 46 & Th. 1 \cite{Hanson1992} & \\
\rowcolor{gray!20}
3 & 11 & 7 & 78 & Th. 7 \cite{Dn-cages} & 78 & Th. 7 \cite{Dn-cages} & \\
3 & 11 & 8 & 100 & Th. 3.3 \cite{biregGirth8} & 102 & Cor. 3.3 \cite{biregGirth8} & \\
\rowcolor{gray!20}
3 & 11 & 9 & 166 & Th. 9 \cite{Dn-cages} & 166 & Th. 9 \cite{Dn-cages} & \\
3 & 11 & 10 & 204 & Th. 2.3 \cite{Aurajo-Pardo2007} & 274 & Cor. 2 \cite{semiCubic2023} & \\
\rowcolor{gray!20}
3 & 11 & 11 & 342 & Th. 1.2 \cite{biregOdd} & 342 & Th. 1.2 \cite{biregOdd} & \\
3 & 11 & 12 & 418 & Th. 2.3 \cite{Aurajo-Pardo2007} & 498 & Cor. 3 \cite{semiCubic2023} & \\
\rowcolor{gray!20}
3 & 11 & 13 & 694 & Th. 1.2 \cite{biregOdd} & 694 & Th. 1.2 \cite{biregOdd} & \\
3 & 11 & 14 & 844 & Th. 2.3 \cite{Aurajo-Pardo2007} & 1 386 & Cor. 4 \cite{semiCubic2023} & \\
3 & 11 & 15 & 1 398 & Moore bound & \underline{3 832} & Th. 2.1 \cite{monotoDn-cages} & \\
3 & 11 & 16 & 1 698 & Th. 2.3 \cite{Aurajo-Pardo2007} & 3 834 & Th. 1 \cite{semiCubic2023} & \\
3 & 11 & 17 & 2 806 & Moore bound & \underline{523 448} $\rightarrow$ 9 002 & Th. 2.1 \cite{monotoDn-cages} & Constr.~\ref{constr:odd} \triAT\\
3 & 11 & 18 & 3 404 & Th. 2.3 \cite{Aurajo-Pardo2007} & 523 450 & Th. 1 \cite{semiCubic2023} & \\
3 & 11 & 19 & 5 622 & Moore bound & \underline{2 095 688} & Th. 2.1 \cite{monotoDn-cages} & \\
3 & 11 & 20 & 6 818 & Th. 2.3 \cite{Aurajo-Pardo2007} & 2 095 690 & Th. 1 \cite{semiCubic2023} & \\
\rowcolor{gray!20}
3 & 12 & 5 & 37 & Cor. 2 \cite{ConstructionsBireg} & 37 & Cor. 2 \cite{ConstructionsBireg} & \\
\rowcolor{gray!20}
3 & 12 & 6 & 50 & Th. 2 \cite{Yang2003} & 50 & Th. 2 \cite{Yang2003} & \\
\rowcolor{gray!20}
3 & 12 & 7 & 85 & Cor. 2 \cite{ConstructionsBireg} & 85 & Cor. 2 \cite{ConstructionsBireg} & \\
\rowcolor{gray!20}
3 & 12 & 8 & 105 & Th. 2.3 \cite{Aurajo-Pardo2007} & 105 & Cor. 7 \cite{ConstructionsBireg} & \\
\rowcolor{gray!20}
3 & 12 & 9 & 181 & Th. 9 \cite{Dn-cages} & 181 & Th. 9 \cite{Dn-cages} & \\
3 & 12 & 10 & 221 & Th. 2.3 \cite{Aurajo-Pardo2007} & 274 $\rightarrow$ 265 & Th. 3 (i) \cite{ConstructionsBireg} & Th.~\ref{thm:genUpperBound}\\
\rowcolor{gray!20}
3 & 12 & 11 & 373 & Cor. 2 \cite{ConstructionsBireg} & 373 & Cor. 2 \cite{ConstructionsBireg} & \\
3 & 12 & 12 & 453 & Th. 2.3 \cite{Aurajo-Pardo2007} & 498 $\rightarrow$ 477 & Th. 3 (i) \cite{ConstructionsBireg} & Th.~\ref{thm:genUpperBound}\\
\rowcolor{gray!20}
3 & 12 & 13 & 757 & Th. 1.2 \cite{biregOdd} & 757 & Th. 1.2 \cite{biregOdd} & \\
3 & 12 & 14 & 917 & Th. 2.3 \cite{Aurajo-Pardo2007} & 1 386 & Cor. 4 \cite{semiCubic2023} & \\
3 & 12 & 15 & 1 525 & Moore bound & 2 474 $\rightarrow$ 2 450 & Th. 3 (iii) \cite{ConstructionsBireg} & Th.~\ref{thm:genUpperBound}\\
3 & 12 & 16 & 1 845 & Th. 2.3 \cite{Aurajo-Pardo2007} & 3 834 $\rightarrow$ 3 774 & Th. 3 (i) \cite{ConstructionsBireg} & Th.~\ref{thm:genUpperBound}\\
3 & 12 & 17 & 3 061 & Moore bound & 262 130 $\rightarrow$ 6 049 & Th. 3 (iii) \cite{ConstructionsBireg} & Constr.~\ref{constr:even} \triAT\\
3 & 12 & 18 & 3 701 & Th. 2.3 \cite{Aurajo-Pardo2007} & 523 450 $\rightarrow$ 10 171 & Th. 3 (i) \cite{ConstructionsBireg} & Th.~\ref{thm:genUpperBound}\\
3 & 12 & 19 & 6 133 & Moore bound & 1 048 562 $\rightarrow$ 17 221 & Th. 3 (iii) \cite{ConstructionsBireg} & Th.~\ref{thm:genUpperBound}\\
3 & 12 & 20 & 7 413 & Th. 2.3 \cite{Aurajo-Pardo2007} & 2 095 690 $\rightarrow$ 21 390 & Th. 3 (i) \cite{ConstructionsBireg} & Th.~\ref{thm:genUpperBound}\\
\rowcolor{gray!20}
3 & 13 & 5 & 40 & Th. 5 \cite{Dn-cages} & 40 & Th. 5 \cite{Dn-cages} & \\
\rowcolor{gray!20}
3 & 13 & 6 & 54 & Th. 1 \cite{Hanson1992} & 54 & Th. 1 \cite{Hanson1992} & \\
\rowcolor{gray!20}
3 & 13 & 7 & 92 & Th. 7 \cite{Dn-cages} & 92 & Th. 7 \cite{Dn-cages} & \\
3 & 13 & 8 & 116 & Th. 3.3 \cite{biregGirth8} & 120 & Th. 3.4 \cite{biregGirth8} & \\
\rowcolor{gray!20}
3 & 13 & 9 & 196 & Th. 9 \cite{Dn-cages} & 196 & Th. 9 \cite{Dn-cages} & \\
3 & 13 & 10 & 240 & Th. 2.3 \cite{Aurajo-Pardo2007} & 344 $\rightarrow$ 338 & Cor. 2 \cite{semiCubic2023} & Th.~\ref{thm:genUpperBound}\\
\rowcolor{gray!20}
3 & 13 & 11 & 404 & Th. 1.2 \cite{biregOdd} & 404 & Th. 1.2 \cite{biregOdd} & \\
3 & 13 & 12 & 490 & Th. 2.3 \cite{Aurajo-Pardo2007} & 624 $\rightarrow$ 606 & Cor. 3 \cite{semiCubic2023} & Th.~\ref{thm:genUpperBound}\\
\rowcolor{gray!20}
3 & 13 & 13 & 820 & Th. 1.2 \cite{biregOdd} & 820 & Th. 1.2 \cite{biregOdd} & \\
3 & 13 & 14 & 992 & Th. 2.3 \cite{Aurajo-Pardo2007} & 1 734 & Cor. 4 \cite{semiCubic2023} & \\
3 & 13 & 15 & 1 652 & Moore bound & \underline{4 792} $\rightarrow$ 3 070 & Th. 2.1 \cite{monotoDn-cages} & Th.~\ref{thm:genUpperBound}\\
3 & 13 & 16 & 1 994 & Th. 2.3 \cite{Aurajo-Pardo2007} & 4 794 $\rightarrow$ 4 734 & Th. 1 \cite{semiCubic2023} & Th.~\ref{thm:genUpperBound}\\
3 & 13 & 17 & 3 316 & Moore bound & \underline{654 312} $\rightarrow$ 10 820 & Th. 2.1 \cite{monotoDn-cages} & Th.~\ref{thm:genUpperBound}\\
\bottomrule
\end{tabular}
}
\caption{Lower and upper bound for $n\rmg$ with $r=3$ and $m$ from 11 till 13.}
\end{table}

\begin{table}[h]
\centering
\scalebox{0.95}{
\begin{tabular}{r r r | r l | r l l}
\toprule
$r$ & $m$ & $g$ & $n\rmg \geq$ & Source & $n\rmg \leq$ & Source & Source improvement \\ 
\midrule
3 & 13 & 18 & 4 000 & Th. 2.3 \cite{Aurajo-Pardo2007} & 654 314 $\rightarrow$ 12 734 & Th. 1 \cite{semiCubic2023} & Th.~\ref{thm:genUpperBound}\\
3 & 13 & 19 & 6 644 & Moore bound & \underline{2 619 612} $\rightarrow$ 21 548 & Th. 2.1 \cite{monotoDn-cages} & Th.~\ref{thm:genUpperBound}\\
3 & 13 & 20 & 8 010 & Th. 2.3 \cite{Aurajo-Pardo2007} & 2 619 614 $\rightarrow$ 26 766 & Th. 1 \cite{semiCubic2023} & Th.~\ref{thm:genUpperBound}\\
\rowcolor{gray!20}
3 & 14 & 5 & 43 & Th. 5 \cite{Dn-cages} & 43 & Th. 5 \cite{Dn-cages} & \\
\rowcolor{gray!20}
3 & 14 & 6 & 58 & Th. 2 \cite{Yang2003} & 58 & Th. 2 \cite{Yang2003} & \\
\rowcolor{gray!20}
3 & 14 & 7 & 99 & Th. 7 \cite{Dn-cages} & 99 & Th. 7 \cite{Dn-cages} & \\
3 & 14 & 8 & 124 & Th. 3.3 \cite{biregGirth8} & 127 & Cor. 3.3 \cite{biregGirth8} & \\
\rowcolor{gray!20}
3 & 14 & 9 & 211 & Th. 9 \cite{Dn-cages} & 211 & Th. 9 \cite{Dn-cages} & \\
3 & 14 & 10 & 257 & Th. 2.3 \cite{Aurajo-Pardo2007} & 342 & Cor. 2 \cite{semiCubic2023} & \\
\rowcolor{gray!20}
3 & 14 & 11 & 435 & Th. 1.2 \cite{biregOdd} & 435 & Th. 1.2 \cite{biregOdd} & \\
3 & 14 & 12 & 526 & Th. 2.3 \cite{Aurajo-Pardo2007} & 622 & Cor. 3 \cite{semiCubic2023} & \\
\rowcolor{gray!20}
3 & 14 & 13 & 883 & Th. 1.2 \cite{biregOdd} & 883 & Th. 1.2 \cite{biregOdd} & \\
3 & 14 & 14 & 1 065 & Th. 2.3 \cite{Aurajo-Pardo2007} & 1 732 & Cor. 4 \cite{semiCubic2023} & \\
3 & 14 & 15 & 1 779 & Moore bound & \underline{4 791} $\rightarrow$ 3 025 & Th. 2.1 \cite{monotoDn-cages} & Constr.~\ref{constr:even} \triAT\\
3 & 14 & 16 & 2 142 & Th. 2.3 \cite{Aurajo-Pardo2007} & 4 792 & Th. 1 \cite{semiCubic2023} & \\
3 & 14 & 17 & 3 571 & Moore bound & \underline{654 311} $\rightarrow$ 8 065 & Th. 2.1 \cite{monotoDn-cages} & Constr.~\ref{constr:even} \triAT\\
3 & 14 & 18 & 4 297 & Th. 2.3 \cite{Aurajo-Pardo2007} & 654 312 & Th. 1 \cite{semiCubic2023} & \\
3 & 14 & 19 & 7 155 & Moore bound & \underline{2 619 611} & Th. 2.1 \cite{monotoDn-cages} & \\
3 & 14 & 20 & 8 606 & Th. 2.3 \cite{Aurajo-Pardo2007} & 2 619 612 & Th. 1 \cite{semiCubic2023} & \\
\rowcolor{gray!20}
3 & 15 & 5 & 46 & Th. 5 \cite{Dn-cages} & 46 & Th. 5 \cite{Dn-cages} & \\
\rowcolor{gray!20}
3 & 15 & 6 & 62 & Th. 1 \cite{Hanson1992} & 62 & Th. 1 \cite{Hanson1992} & \\
\rowcolor{gray!20}
3 & 15 & 7 & 106 & Th. 7 \cite{Dn-cages} & 106 & Th. 7 \cite{Dn-cages} & \\
\rowcolor{gray!20}
3 & 15 & 8 & 130 & Th. 2.3 \cite{Aurajo-Pardo2007} & 130 & Cor. 7 \cite{ConstructionsBireg} & \\
\rowcolor{gray!20}
3 & 15 & 9 & 226 & Th. 9 \cite{Dn-cages} & 226 & Th. 9 \cite{Dn-cages} & \\
3 & 15 & 10 & 274 & Th. 2.3 \cite{Aurajo-Pardo2007} & 342 $\rightarrow$ 330 & Th. 3 (i) \cite{ConstructionsBireg} & Th.~\ref{thm:genUpperBound}\\
\rowcolor{gray!20}
3 & 15 & 11 & 466 & Th. 1.2 \cite{biregOdd} & 466 & Th. 1.2 \cite{biregOdd} & \\
3 & 15 & 12 & 562 & Th. 2.3 \cite{Aurajo-Pardo2007} & 622 $\rightarrow$ 594 & Th. 3 (i) \cite{ConstructionsBireg} & Th.~\ref{thm:genUpperBound}\\
\rowcolor{gray!20}
3 & 15 & 13 & 946 & Th. 1.2 \cite{biregOdd} & 946 & Th. 1.2 \cite{biregOdd} & \\
3 & 15 & 14 & 1 138 & Th. 2.3 \cite{Aurajo-Pardo2007} & 1 732 & Cor. 4 \cite{semiCubic2023} & \\
3 & 15 & 15 & 1 906 & Moore bound & 3 092 $\rightarrow$ 3 060 & Th. 3 (iii) \cite{ConstructionsBireg} & Th.~\ref{thm:genUpperBound}\\
3 & 15 & 16 & 2 290 & Th. 2.3 \cite{Aurajo-Pardo2007} & 4 792 $\rightarrow$ 4 712 & Th. 3 (i) \cite{ConstructionsBireg} & Th.~\ref{thm:genUpperBound}\\
3 & 15 & 17 & 3 826 & Moore bound & 327 662 $\rightarrow$ 10 800 & Th. 3 (iii) \cite{ConstructionsBireg} & Th.~\ref{thm:genUpperBound}\\
3 & 15 & 18 & 4 594 & Th. 2.3 \cite{Aurajo-Pardo2007} & 654 312 $\rightarrow$ 12 708 & Th. 3 (i) \cite{ConstructionsBireg} & Th.~\ref{thm:genUpperBound}\\
3 & 15 & 19 & 7 666 & Moore bound & 1 310 702 $\rightarrow$ 21 520 & Th. 3 (iii) \cite{ConstructionsBireg} & Th.~\ref{thm:genUpperBound}\\
3 & 15 & 20 & 9 202 & Th. 2.3 \cite{Aurajo-Pardo2007} & 2 619 612 $\rightarrow$ 26 728 & Th. 3 (i) \cite{ConstructionsBireg} & Th.~\ref{thm:genUpperBound}\\
\rowcolor{gray!20}
3 & 16 & 5 & 49 & Cor. 2 \cite{ConstructionsBireg} & 49 & Cor. 2 \cite{ConstructionsBireg} & \\
\rowcolor{gray!20}
3 & 16 & 6 & 66 & Th. 1 \cite{Hanson1992} & 66 & Th. 1 \cite{Hanson1992} & \\
\rowcolor{gray!20}
3 & 16 & 7 & 113 & Cor. 2 \cite{ConstructionsBireg} & 113 & Cor. 2 \cite{ConstructionsBireg} & \\
3 & 16 & 8 & 141 & Th. 3.3 \cite{biregGirth8} & 147 & Th. 3.4 \cite{biregGirth8} & \\
\rowcolor{gray!20}
3 & 16 & 9 & 241 & Th. 9 \cite{Dn-cages} & 241 & Th. 9 \cite{Dn-cages} & \\
3 & 16 & 10 & 292 & Th. 2.3 \cite{Aurajo-Pardo2007} & 412 $\rightarrow$ 404 & Cor. 2 \cite{semiCubic2023} & Th.~\ref{thm:genUpperBound}\\
\rowcolor{gray!20}
3 & 16 & 11 & 497 & Cor. 2 \cite{ConstructionsBireg} & 497 & Cor. 2 \cite{ConstructionsBireg} & \\
3 & 16 & 12 & 599 & Th. 2.3 \cite{Aurajo-Pardo2007} & 748 $\rightarrow$ 724 & Cor. 3 \cite{semiCubic2023} & Th.~\ref{thm:genUpperBound}\\
\rowcolor{gray!20}
3 & 16 & 13 & 1 009 & Th. 1.2 \cite{biregOdd} & 1 009 & Th. 1.2 \cite{biregOdd} & \\
\bottomrule
\end{tabular}
}
\caption{Lower and upper bound for $n\rmg$ with $r=3$ and $m$ from 13 till 16.}
\end{table}

\begin{table}[h]
\centering
\scalebox{0.95}{
\begin{tabular}{r r r | r l | r l l}
\toprule
$r$ & $m$ & $g$ & $n\rmg \geq$ & Source & $n\rmg \leq$ & Source & Source improvement \\ 
\midrule
3 & 16 & 14 & 1 212 & Th. 2.3 \cite{Aurajo-Pardo2007} & 2 080 & Cor. 4 \cite{semiCubic2023} & \\
3 & 16 & 15 & 2 033 & Moore bound & 3 833 $\rightarrow$ 3 457 & Th. 1 \cite{ConstructionsBireg} & Constr.~\ref{constr:even} \triAT\\
3 & 16 & 16 & 2 439 & Th. 2.3 \cite{Aurajo-Pardo2007} & 5 752 $\rightarrow$ 4 961 & Th. 1 \cite{semiCubic2023} & Constr.~\ref{constr:even} \triAT\\
3 & 16 & 17 & 4 081 & Moore bound & 523 449 $\rightarrow$ 8 065 & Th. 1 \cite{ConstructionsBireg} & Constr.~\ref{constr:even} \triAT\\
3 & 16 & 18 & 4 892 & Th. 2.3 \cite{Aurajo-Pardo2007} & 785 176 $\rightarrow$ 15 272 & Th. 1 \cite{semiCubic2023} & Th.~\ref{thm:genUpperBound}\\
3 & 16 & 19 & 8 177 & Moore bound & 2 095 689 $\rightarrow$ 25 848 & Th. 1 \cite{ConstructionsBireg} & Th.~\ref{thm:genUpperBound}\\
3 & 16 & 20 & 9 799 & Th. 2.3 \cite{Aurajo-Pardo2007} & 3 143 536 $\rightarrow$ 32 104 & Th. 1 \cite{semiCubic2023} & Th.~\ref{thm:genUpperBound}\\
\rowcolor{gray!20}
3 & 17 & 5 & 52 & Th. 5 \cite{Dn-cages} & 52 & Th. 5 \cite{Dn-cages} & \\
\rowcolor{gray!20}
3 & 17 & 6 & 70 & Th. 2 \cite{Yang2003} & 70 & Th. 2 \cite{Yang2003} & \\
\rowcolor{gray!20}
3 & 17 & 7 & 120 & Th. 7 \cite{Dn-cages} & 120 & Th. 7 \cite{Dn-cages} & \\
3 & 17 & 8 & 150 & Th. 3.3 \cite{biregGirth8} & 152 & Cor. 3.3 \cite{biregGirth8} & \\
\rowcolor{gray!20}
3 & 17 & 9 & 256 & Th. 9 \cite{Dn-cages} & 256 & Th. 9 \cite{Dn-cages} & \\
3 & 17 & 10 & 310 & Th. 2.3 \cite{Aurajo-Pardo2007} & 410 & Cor. 2 \cite{semiCubic2023} & \\
\rowcolor{gray!20}
3 & 17 & 11 & 528 & Th. 1.2 \cite{biregOdd} & 528 & Th. 1.2 \cite{biregOdd} & \\
3 & 17 & 12 & 636 & Th. 2.3 \cite{Aurajo-Pardo2007} & 746 & Cor. 3 \cite{semiCubic2023} & \\
\rowcolor{gray!20}
3 & 17 & 13 & 1 072 & Th. 1.2 \cite{biregOdd} & 1 072 & Th. 1.2 \cite{biregOdd} & \\
3 & 17 & 14 & 1 286 & Th. 2.3 \cite{Aurajo-Pardo2007} & 2 078 & Cor. 4 \cite{semiCubic2023} & \\
3 & 17 & 15 & 2 160 & Moore bound & \underline{5 748} & Th. 2.1 \cite{monotoDn-cages} & \\
3 & 17 & 16 & 2 588 & Th. 2.3 \cite{Aurajo-Pardo2007} & 5 750 & Th. 1 \cite{semiCubic2023} & \\
3 & 17 & 17 & 4 336 & Moore bound & \underline{785 172} & Th. 2.1 \cite{monotoDn-cages} & \\
3 & 17 & 18 & 5 190 & Th. 2.3 \cite{Aurajo-Pardo2007} & 785 174 & Th. 1 \cite{semiCubic2023} & \\
3 & 17 & 19 & 8 688 & Moore bound & \underline{3 143 532} & Th. 2.1 \cite{monotoDn-cages} & \\
3 & 17 & 20 & 10 396 & Th. 2.3 \cite{Aurajo-Pardo2007} & 3 143 534 & Th. 1 \cite{semiCubic2023} & \\
\rowcolor{gray!20}
3 & 18 & 5 & 55 & Th. 5 \cite{Dn-cages} & 55 & Th. 5 \cite{Dn-cages} & \\
\rowcolor{gray!20}
3 & 18 & 6 & 74 & Th. 2 \cite{Yang2003} & 74 & Th. 2 \cite{Yang2003} & \\
\rowcolor{gray!20}
3 & 18 & 7 & 127 & Th. 7 \cite{Dn-cages} & 127 & Th. 7 \cite{Dn-cages} & \\
\rowcolor{gray!20}
3 & 18 & 8 & 155 & Th. 2.3 \cite{Aurajo-Pardo2007} & 155 & Cor. 7 \cite{ConstructionsBireg} & \\
\rowcolor{gray!20}
3 & 18 & 9 & 271 & Th. 9 \cite{Dn-cages} & 271 & Th. 9 \cite{Dn-cages} & \\
3 & 18 & 10 & 327 & Th. 2.3 \cite{Aurajo-Pardo2007} & 410 $\rightarrow$ 395 & Th. 3 (i) \cite{ConstructionsBireg} & Th.~\ref{thm:genUpperBound}\\
\rowcolor{gray!20}
3 & 18 & 11 & 559 & Th. 1.2 \cite{biregOdd} & 559 & Th. 1.2 \cite{biregOdd} & \\
3 & 18 & 12 & 671 & Th. 2.3 \cite{Aurajo-Pardo2007} & 746 $\rightarrow$ 711 & Th. 3 (i) \cite{ConstructionsBireg} & Th.~\ref{thm:genUpperBound}\\
\rowcolor{gray!20}
3 & 18 & 13 & 1 135 & Th. 1.2 \cite{biregOdd} & 1 135 & Th. 1.2 \cite{biregOdd} & \\
3 & 18 & 14 & 1 359 & Th. 2.3 \cite{Aurajo-Pardo2007} & 2 078 & Cor. 4 \cite{semiCubic2023} & \\
3 & 18 & 15 & 2 287 & Moore bound & 3 710 $\rightarrow$ 3 670 & Th. 3 (iii) \cite{ConstructionsBireg} & Th.~\ref{thm:genUpperBound}\\
3 & 18 & 16 & 2 735 & Th. 2.3 \cite{Aurajo-Pardo2007} & 5 750 $\rightarrow$ 5 650 & Th. 3 (i) \cite{ConstructionsBireg} & Th.~\ref{thm:genUpperBound}\\
3 & 18 & 17 & 4 591 & Moore bound & 393 194 $\rightarrow$ 9 073 & Th. 3 (iii) \cite{ConstructionsBireg} & Constr.~\ref{constr:even} \triAT\\
3 & 18 & 18 & 5 487 & Th. 2.3 \cite{Aurajo-Pardo2007} & 785 174 $\rightarrow$ 15 245 & Th. 3 (i) \cite{ConstructionsBireg} & Th.~\ref{thm:genUpperBound}\\
3 & 18 & 19 & 9 199 & Moore bound & 1 572 842 $\rightarrow$ 25 819 & Th. 3 (iii) \cite{ConstructionsBireg} & Th.~\ref{thm:genUpperBound}\\
3 & 18 & 20 & 10 991 & Th. 2.3 \cite{Aurajo-Pardo2007} & 3 143 534 $\rightarrow$ 32 066 & Th. 3 (i) \cite{ConstructionsBireg} & Th.~\ref{thm:genUpperBound}\\
\bottomrule
\end{tabular}
}
\caption{Lower and upper bound for $n\rmg$ with $r=3$ and $m$ from 16 till 18.}
\end{table}

\begin{table}[h]
\centering
\scalebox{0.95}{
\begin{tabular}{r r r | r l | r l l}
\toprule
$r$ & $m$ & $g$ & $n\rmg \geq$ & Source & $n\rmg \leq$ & Source & Source improvement \\ 
\midrule
\rowcolor{gray!20}
4 & 5 & 5 & 21 & Th. 2 \cite{Hanson1992} & 21 & Th. 2 \cite{Hanson1992} & \\
\rowcolor{gray!20}
4 & 5 & 6 & 32 & Th. 2 \cite{Yang2003} & 32 & Th. 2 \cite{Yang2003} & \\
4 & 5 & 7 & 66 $\rightarrow$ 69 & Moore bound & \underline{120} $\rightarrow$ 96 & Th. 2.1 \cite{monotoDn-cages} & Constr.~\ref{constr:addEdge} \tetraAT\\
4 & 5 & 8 & 97 & Th. 2.3 \cite{Aurajo-Pardo2007} & 121 & Th. 3.1 \cite{semiRegCages} & \\
4 & 5 & 9 & 201 & Moore bound & \underline{1 125} $\rightarrow$ 360 & Th. 2.1 \cite{monotoDn-cages} & Constr.~\ref{constr:addEdge} \tetraAT\\
4 & 5 & 10 & 295 & Th. 2.3 \cite{Aurajo-Pardo2007} & 1 126 $\rightarrow$ 514 & Th. 3.1 \cite{semiRegCages} & Constr.~\ref{constr:odd} \tetraAT\\
4 & 5 & 11 & 606 & Moore bound & \underline{1 920} $\rightarrow$ 1 260 & Th. 2.1 \cite{monotoDn-cages} & Constr.~\ref{constr:addEdge} \tetraTwoAT\\
4 & 5 & 12 & 889 & Th. 2.3 \cite{Aurajo-Pardo2007} & 1 921 & Th. 3.1 \cite{semiRegCages} & \\
\rowcolor{gray!20}
4 & 6 & 5 & 25 & Cor. 2 \cite{ConstructionsBireg} & 25 & Cor. 2 \cite{ConstructionsBireg} & \\
\rowcolor{gray!20}
4 & 6 & 6 & 38 & Th. 2 \cite{Yang2003} & 38 & Th. 2 \cite{Yang2003} & \\
\rowcolor{gray!20}
4 & 6 & 7 & 79 & Cor. 2 \cite{ConstructionsBireg} & 79 & Cor. 2 \cite{ConstructionsBireg} & \\
4 & 6 & 8 & 114 & Th. 2.3 \cite{Aurajo-Pardo2007} & 138 & Th. 2.2 \cite{biregGirth8} & \\
4 & 6 & 9 & 241 & Moore bound & 383 & Th. 1 \cite{ConstructionsBireg} & \\
4 & 6 & 10 & 348 & Th. 2.3 \cite{Aurajo-Pardo2007} & 726 & Th. 1 \cite{ConstructionsBireg} & \\
\rowcolor{gray!20}
4 & 6 & 11 & 727 & Cor. 2 \cite{ConstructionsBireg} & 727 & Cor. 2 \cite{ConstructionsBireg} & \\
4 & 6 & 12 & 1 050 & Th. 2.3 \cite{Aurajo-Pardo2007} & 1 386 & Th. 2.2 \cite{biregGirth8} & \\
\rowcolor{gray!20}
4 & 7 & 5 & 29 & Th. 2 \cite{Hanson1992} & 29 & Th. 2 \cite{Hanson1992} & \\
\rowcolor{gray!20}
4 & 7 & 6 & 44 & Th. 3 \cite{Yang2003} & 44 & Th. 3 \cite{Yang2003} & \\
4 & 7 & 7 & 92 & Moore bound & \underline{142} & Th. 2.1 \cite{monotoDn-cages} & \\
4 & 7 & 8 & 131 & Th. 2.3 \cite{Aurajo-Pardo2007} & 143 & Th. 2.2 \cite{biregGirth8} & \\
4 & 7 & 9 & 281 & Moore bound & \underline{741} $\rightarrow$ 422 & Th. 2.1 \cite{monotoDn-cages} & Constr.~\ref{constr:odd} \tetraAT\\
4 & 7 & 10 & 401 & Th. 2.3 \cite{Aurajo-Pardo2007} & 742 & Th. 2.2 (i) \cite{Aurajo-Pardo2007} & \\
4 & 7 & 11 & 848 & Moore bound & \underline{1 402} & Th. 2.1 \cite{monotoDn-cages} & \\
4 & 7 & 12 & 1 211 & Th. 2.3 \cite{Aurajo-Pardo2007} & 1 403 & Th. 2.2 \cite{biregGirth8} & \\
\rowcolor{gray!20}
4 & 8 & 5 & 33 & Th. 2 \cite{Hanson1992} & 33 & Th. 2 \cite{Hanson1992} & \\
\rowcolor{gray!20}
4 & 8 & 6 & 50 & Th. 2 \cite{Yang2003} & 50 & Th. 2 \cite{Yang2003} & \\
\rowcolor{gray!20}
4 & 8 & 7 & 105 & Th. 1.2 \cite{biregOdd} & 105 & Th. 1.2 \cite{biregOdd} & \\
4 & 8 & 8 & 148 & Th. 2.3 \cite{Aurajo-Pardo2007} & 151 & Cor. 6 (ii) \cite{ConstructionsBireg} & \\
4 & 8 & 9 & 321 & Moore bound & 510 $\rightarrow$ 505 & Th. 3 (iii) \cite{ConstructionsBireg} & Constr.~\ref{constr:even} \tetraAT\\
4 & 8 & 10 & 454 & Th. 2.3 \cite{Aurajo-Pardo2007} & 766 & Th. 3 (i) \cite{ConstructionsBireg} & \\
4 & 8 & 11 & 969 & Moore bound & \underline{1 428} & Th. 2.1 \cite{monotoDn-cages} & \\
4 & 8 & 12 & 1 372 & Th. 2.3 \cite{Aurajo-Pardo2007} & 1 429 & Cor. 6 (ii) \cite{ConstructionsBireg} & \\
\rowcolor{gray!20}
4 & 9 & 5 & 37 & Th. 2 \cite{Hanson1992} & 37 & Th. 2 \cite{Hanson1992} & \\
\rowcolor{gray!20}
4 & 9 & 6 & 56 & Th. 2 \cite{Yang2003} & 56 & Th. 2 \cite{Yang2003} & \\
4 & 9 & 7 & 118 & Moore bound & \underline{200} $\rightarrow$ 191 & Th. 2.1 \cite{monotoDn-cages} & Th.~\ref{thm:genUpperBound}\\
4 & 9 & 8 & 165 & Th. 2.3 \cite{Aurajo-Pardo2007} & 201 & Th. 2.2 \cite{biregGirth8} & \\
4 & 9 & 9 & 361 & Moore bound & \underline{1 091} $\rightarrow$ 809 & Th. 2.1 \cite{monotoDn-cages} & Th.~\ref{thm:genUpperBound}\\
4 & 9 & 10 & 507 & Th. 2.3 \cite{Aurajo-Pardo2007} & 1 092 $\rightarrow$ 962 & Th. 2.2 (ii) \cite{Aurajo-Pardo2007} & Constr.~\ref{constr:odd} \tetraTwoAT\\
4 & 9 & 11 & 1 090 & Moore bound & \underline{2 060} & Th. 2.1 \cite{monotoDn-cages} & \\
4 & 9 & 12 & 1 533 & Th. 2.3 \cite{Aurajo-Pardo2007} & 2 061 & Th. 2.2 \cite{biregGirth8} & \\
\rowcolor{gray!20}
4 & 10 & 5 & 41 & Th. 2 \cite{Hanson1992} & 41 & Th. 2 \cite{Hanson1992} & \\
\rowcolor{gray!20}
4 & 10 & 6 & 62 & Th. 2 \cite{Yang2003} & 62 & Th. 2 \cite{Yang2003} & \\
\rowcolor{gray!20}
4 & 10 & 7 & 131 & Th. 1.2 \cite{biregOdd} & 131 & Th. 1.2 \cite{biregOdd} & \\
4 & 10 & 8 & 182 & Th. 2.3 \cite{Aurajo-Pardo2007} & 206 & Th. 2.2 \cite{biregGirth8} & \\
\bottomrule
\end{tabular}
}
\caption{Lower and upper bound for $n\rmg$ with $r=4$ and $m$ from 5 till 10.}
\end{table}

\begin{table}[h]
\centering
\scalebox{0.95}{
\begin{tabular}{r r r | r l | r l l}
\toprule
$r$ & $m$ & $g$ & $n\rmg \geq$ & Source & $n\rmg \leq$ & Source & Source improvement \\ 
\midrule
4 & 10 & 9 & 401 & Moore bound & \underline{1 099} $\rightarrow$ 961 & Th. 2.1 \cite{monotoDn-cages} & Constr.~\ref{constr:even} \tetraTwoAT\\
4 & 10 & 10 & 560 & Th. 2.3 \cite{Aurajo-Pardo2007} & 1 100 & Th. 2.2 (i) \cite{Aurajo-Pardo2007} & \\
4 & 10 & 11 & 1 211 & Moore bound & \underline{2 077} & Th. 2.1 \cite{monotoDn-cages} & \\
4 & 10 & 12 & 1 694 & Th. 2.3 \cite{Aurajo-Pardo2007} & 2 078 & Th. 2.2 \cite{biregGirth8} & \\
\rowcolor{gray!20}
4 & 11 & 5 & 45 & Th. 2 \cite{Hanson1992} & 45 & Th. 2 \cite{Hanson1992} & \\
\rowcolor{gray!20}
4 & 11 & 6 & 68 & Th. 3 \cite{Yang2003} & 68 & Th. 3 \cite{Yang2003} & \\
4 & 11 & 7 & 144 & Moore bound & 118 162 $\rightarrow$ 242 &  \cite{Chartrand1981} & Constr.~\ref{constr:odd} \tetraAT\\
4 & 11 & 8 & 199 & Th. 2.3 \cite{Aurajo-Pardo2007} & 236 252 $\rightarrow$ 302 &  \cite{Chartrand1981} & Constr.~\ref{constr:odd} \tetraET\\
4 & 11 & 9 & 441 & Moore bound & 9 566 194 $\rightarrow$ 1 094 &  \cite{Chartrand1981} & Constr.~\ref{constr:odd} \tetraTwoAT\\
4 & 11 & 10 & 613 & Th. 2.3 \cite{Aurajo-Pardo2007} & 19 132 184 $\rightarrow$ 1 322 &  \cite{Chartrand1981} & Constr.~\ref{constr:odd} \tetraTwoAT\\
4 & 11 & 11 & 1 332 & Moore bound & 774 841 705 &  \cite{Chartrand1981} & \\
4 & 11 & 12 & 1 855 & Th. 2.3 \cite{Aurajo-Pardo2007} & 1 549 682 436 &  \cite{Chartrand1981} & \\
\rowcolor{gray!20}
4 & 12 & 5 & 49 & Cor. 2 \cite{ConstructionsBireg} & 49 & Cor. 2 \cite{ConstructionsBireg} & \\
\rowcolor{gray!20}
4 & 12 & 6 & 74 & Th. 2 \cite{Yang2003} & 74 & Th. 2 \cite{Yang2003} & \\
\rowcolor{gray!20}
4 & 12 & 7 & 157 & Cor. 2 \cite{ConstructionsBireg} & 157 & Cor. 2 \cite{ConstructionsBireg} & \\
4 & 12 & 8 & 216 & Th. 2.3 \cite{Aurajo-Pardo2007} & 222 & Cor. 6 (ii) \cite{ConstructionsBireg} & \\
4 & 12 & 9 & 481 & Moore bound & 764 & Th. 3 (iii) \cite{ConstructionsBireg} & \\
4 & 12 & 10 & 666 & Th. 2.3 \cite{Aurajo-Pardo2007} & 1 148 & Th. 3 (i) \cite{ConstructionsBireg} & \\
\rowcolor{gray!20}
4 & 12 & 11 & 1 453 & Cor. 2 \cite{ConstructionsBireg} & 1 453 & Cor. 2 \cite{ConstructionsBireg} & \\
4 & 12 & 12 & 2 016 & Th. 2.3 \cite{Aurajo-Pardo2007} & 2 130 & Cor. 6 (ii) \cite{ConstructionsBireg} & \\
\rowcolor{gray!20}
4 & 13 & 5 & 53 & Th. 2 \cite{Hanson1992} & 53 & Th. 2 \cite{Hanson1992} & \\
\rowcolor{gray!20}
4 & 13 & 6 & 80 & Th. 3 \cite{Yang2003} & 80 & Th. 3 \cite{Yang2003} & \\
4 & 13 & 7 & 170 & Moore bound & \underline{268} & Th. 2.1 \cite{monotoDn-cages} & \\
4 & 13 & 8 & 233 & Th. 2.3 \cite{Aurajo-Pardo2007} & 269 & Th. 2.2 \cite{biregGirth8} & \\
4 & 13 & 9 & 521 & Moore bound & \underline{1 457} $\rightarrow$ 1 351 & Th. 2.1 \cite{monotoDn-cages} & Th.~\ref{thm:genUpperBound}\\
4 & 13 & 10 & 719 & Th. 2.3 \cite{Aurajo-Pardo2007} & 1 458 & Th. 2.2 (i) \cite{Aurajo-Pardo2007} & \\
4 & 13 & 11 & 1 574 & Moore bound & \underline{2 752} & Th. 2.1 \cite{monotoDn-cages} & \\
4 & 13 & 12 & 2 177 & Th. 2.3 \cite{Aurajo-Pardo2007} & 2 753 & Th. 2.2 \cite{biregGirth8} & \\
\bottomrule
\end{tabular}
}
\caption{Lower and upper bound for $n\rmg$ with $r=4$ and $m$ from 10 till 13.}
\end{table}

\begin{table}[h]
\centering
\scalebox{0.95}{
\begin{tabular}{r r r | r l | r l l}
\toprule
$r$ & $m$ & $g$ & $n\rmg \geq$ & Source & $n\rmg \leq$ & Source & Source improvement \\ 
\midrule
\rowcolor{gray!20}
5 & 6 & 5 & 31 & Th. 1.2 \cite{biregOdd} & 31 & Th. 1.2 \cite{biregOdd} & \\
\rowcolor{gray!20}
5 & 6 & 6 & 50 & Th. 2 \cite{Yang2003} & 50 & Th. 2 \cite{Yang2003} & \\
5 & 6 & 7 & 127 & Moore bound & \underline{240} & Th. 2.1 \cite{monotoDn-cages} & \\
5 & 6 & 8 & 196 & Th. 2.3 \cite{Aurajo-Pardo2007} & 241 & Th. 3.1 \cite{semiRegCages} & \\
5 & 6 & 9 & 511 & Moore bound & \underline{5 998} $\rightarrow$ \underline{2 556} & Th. 2.1 \cite{monotoDn-cages} & \\
5 & 6 & 10 & 788 & Th. 2.3 \cite{Aurajo-Pardo2007} & \underline{5 999} $\rightarrow$ 2 557 & Th. 2.1 \cite{monotoDn-cages} & Th.~\ref{thm:genUpperBound}\\
5 & 6 & 11 & 2 047 & Moore bound & \underline{6 000} $\rightarrow$ 5 351 & Th. 2.1 \cite{monotoDn-cages} & Th.~\ref{thm:genUpperBound}\\
5 & 6 & 12 & 3 156 & Th. 2.3 \cite{Aurajo-Pardo2007} & 6 001 $\rightarrow$ 5 414 & Th. 3.1 \cite{semiRegCages} & Th.~\ref{thm:genUpperBound}\\
\rowcolor{gray!20}
5 & 7 & 5 & 36 & Th. 1.2 \cite{biregOdd} & 36 & Th. 1.2 \cite{biregOdd} & \\
\rowcolor{gray!20}
5 & 7 & 6 & 58 & Th. 3 \cite{Yang2003} & 58 & Th. 3 \cite{Yang2003} & \\
5 & 7 & 7 & 148 & Moore bound & 822 $\rightarrow$ 290 &  \cite{Chartrand1981} & Constr.~\ref{constr:odd} \pentaAT\\
5 & 7 & 8 & 222 & Th. 2.3 \cite{Aurajo-Pardo2007} & 842 $\rightarrow$ 386 &  \cite{Chartrand1981} & Constr.~\ref{constr:odd} \pentaAT\\
5 & 7 & 9 & 596 & Moore bound & 34 218 $\rightarrow$ \underline{3 852} &  \cite{Chartrand1981} & \\
5 & 7 & 10 & 894 & Th. 2.3 \cite{Aurajo-Pardo2007} & 34 222 $\rightarrow$ 3 854 &  \cite{Chartrand1981} & Th.~\ref{thm:genUpperBound}\\
5 & 7 & 11 & 2 388 & Moore bound & 35 614 $\rightarrow$ 8 040 &  \cite{Chartrand1981} & Th.~\ref{thm:genUpperBound}\\
5 & 7 & 12 & 3 582 & Th. 2.3 \cite{Aurajo-Pardo2007} & 35 658 $\rightarrow$ 8 144 &  \cite{Chartrand1981} & Th.~\ref{thm:genUpperBound}\\
\rowcolor{gray!20}
5 & 8 & 5 & 41 & Cor. 2 \cite{ConstructionsBireg} & 41 & Cor. 2 \cite{ConstructionsBireg} & \\
\rowcolor{gray!20}
5 & 8 & 6 & 66 & Th. 2 \cite{Yang2003} & 66 & Th. 2 \cite{Yang2003} & \\
\rowcolor{gray!20}
5 & 8 & 7 & 169 & Cor. 2 \cite{ConstructionsBireg} & 169 & Cor. 2 \cite{ConstructionsBireg} & \\
5 & 8 & 8 & 248 & Th. 2.3 \cite{Aurajo-Pardo2007} & 308 & Th. 2.2 \cite{biregGirth8} & \\
5 & 8 & 9 & 681 & Moore bound & 1 295 & Th. 1 \cite{ConstructionsBireg} & \\
5 & 8 & 10 & 1 000 & Th. 2.3 \cite{Aurajo-Pardo2007} & 2 540 & Th. 2.2 (ii) \cite{Aurajo-Pardo2007} & \\
\rowcolor{gray!20}
5 & 8 & 11 & 2 729 & Cor. 2 \cite{ConstructionsBireg} & 2 729 & Cor. 2 \cite{ConstructionsBireg} & \\
5 & 8 & 12 & 4 008 & Th. 2.3 \cite{Aurajo-Pardo2007} & 5 328 & Th. 2.2 \cite{biregGirth8} & \\
\rowcolor{gray!20}
5 & 9 & 5 & 46 & Th. 1.2 \cite{biregOdd} & 46 & Th. 1.2 \cite{biregOdd} & \\
\rowcolor{gray!20}
5 & 9 & 6 & 74 & Th. 2 \cite{Yang2003} & 74 & Th. 2 \cite{Yang2003} & \\
5 & 9 & 7 & 190 & Moore bound & \underline{312} $\rightarrow$ 290 & Th. 2.1 \cite{monotoDn-cages} & Constr.~\ref{constr:odd} \pentaAT\\
5 & 9 & 8 & 274 & Th. 2.3 \cite{Aurajo-Pardo2007} & 314 & Th. 2.2 \cite{biregGirth8} & \\
5 & 9 & 9 & 766 & Moore bound & \underline{2 548} & Th. 2.1 \cite{monotoDn-cages} & \\
5 & 9 & 10 & 1 106 & Th. 2.3 \cite{Aurajo-Pardo2007} & 2 550 & Th. 2.2 (i) \cite{Aurajo-Pardo2007} & \\
5 & 9 & 11 & 3 070 & Moore bound & \underline{5 352} & Th. 2.1 \cite{monotoDn-cages} & \\
5 & 9 & 12 & 4 434 & Th. 2.3 \cite{Aurajo-Pardo2007} & 5 354 & Th. 2.2 \cite{biregGirth8} & \\
\rowcolor{gray!20}
5 & 10 & 5 & 51 & Th. 1.2 \cite{biregOdd} & 51 & Th. 1.2 \cite{biregOdd} & \\
\rowcolor{gray!20}
5 & 10 & 6 & 82 & Th. 2 \cite{Yang2003} & 82 & Th. 2 \cite{Yang2003} & \\
5 & 10 & 7 & 211 & Moore bound & 302 & Th. 3 (iii) \cite{ConstructionsBireg} & \\
5 & 10 & 8 & 300 & Th. 2.3 \cite{Aurajo-Pardo2007} & 324 & Cor. 6 (ii) \cite{ConstructionsBireg} & \\
5 & 10 & 9 & 851 & Moore bound & 2 586 & Th. 3 (iii) \cite{ConstructionsBireg} & \\
5 & 10 & 10 & 1 212 & Th. 2.3 \cite{Aurajo-Pardo2007} & 2 590 & Th. 3 (i) \cite{ConstructionsBireg} & \\
5 & 10 & 11 & 3 411 & Moore bound & 5 374 & Th. 3 (iii) \cite{ConstructionsBireg} & \\
5 & 10 & 12 & 4 860 & Th. 2.3 \cite{Aurajo-Pardo2007} & 5 458 & Th. 3 (i) \cite{ConstructionsBireg} & \\
\rowcolor{gray!20}
5 & 11 & 5 & 56 & Th. 1.2 \cite{biregOdd} & 56 & Th. 1.2 \cite{biregOdd} & \\
\rowcolor{gray!20}
5 & 11 & 6 & 90 & Th. 3 \cite{Yang2003} & 90 & Th. 3 \cite{Yang2003} & \\
5 & 11 & 7 & 232 & Moore bound & 118 250 $\rightarrow$ 322 &  \cite{Chartrand1981} & Constr.~\ref{constr:odd} \pentaAT\\
5 & 11 & 8 & 326 & Th. 2.3 \cite{Aurajo-Pardo2007} & 236 342 &  \cite{Chartrand1981} & \\
\bottomrule
\end{tabular}
}
\caption{Lower and upper bound for $n\rmg$ with $r=5$ and $m$ from 6 till 11.}
\end{table}

\begin{table}[h]
\centering
\scalebox{0.95}{
\begin{tabular}{r r r | r l | r l l}
\toprule
$r$ & $m$ & $g$ & $n\rmg \geq$ & Source & $n\rmg \leq$ & Source & Source improvement \\ 
\midrule
5 & 11 & 9 & 936 & Moore bound & 9 567 232 &  \cite{Chartrand1981} & \\
5 & 11 & 10 & 1 318 & Th. 2.3 \cite{Aurajo-Pardo2007} & 19 133 096 &  \cite{Chartrand1981} & \\
5 & 11 & 11 & 3 752 & Moore bound & 774 843 666 &  \cite{Chartrand1981} & \\
5 & 11 & 12 & 5 286 & Th. 2.3 \cite{Aurajo-Pardo2007} & 1 549 684 438 &  \cite{Chartrand1981} & \\
\rowcolor{gray!20}
5 & 12 & 5 & 61 & Th. 1.2 \cite{biregOdd} & 61 & Th. 1.2 \cite{biregOdd} & \\
\rowcolor{gray!20}
5 & 12 & 6 & 98 & Th. 2 \cite{Yang2003} & 98 & Th. 2 \cite{Yang2003} & \\
5 & 12 & 7 & 253 & Moore bound & \underline{451} $\rightarrow$ 385 & Th. 2.1 \cite{monotoDn-cages} & Constr.~\ref{constr:even} \pentaAT\\
5 & 12 & 8 & 352 & Th. 2.3 \cite{Aurajo-Pardo2007} & 452 & Th. 2.2 \cite{biregGirth8} & \\
5 & 12 & 9 & 1 021 & Moore bound & \underline{3 793} & Th. 2.1 \cite{monotoDn-cages} & \\
5 & 12 & 10 & 1 424 & Th. 2.3 \cite{Aurajo-Pardo2007} & 3 794 & Th. 2.2 (ii) \cite{Aurajo-Pardo2007} & \\
5 & 12 & 11 & 4 093 & Moore bound & \underline{7 951} & Th. 2.1 \cite{monotoDn-cages} & \\
5 & 12 & 12 & 5 712 & Th. 2.3 \cite{Aurajo-Pardo2007} & 7 952 & Th. 2.2 \cite{biregGirth8} & \\
\rowcolor{gray!20}
5 & 13 & 5 & 66 & Th. 1.2 \cite{biregOdd} & 66 & Th. 1.2 \cite{biregOdd} & \\
\rowcolor{gray!20}
5 & 13 & 6 & 106 & Th. 3 \cite{Yang2003} & 106 & Th. 3 \cite{Yang2003} & \\
5 & 13 & 7 & 274 & Moore bound & \underline{456} $\rightarrow$ 386 & Th. 2.1 \cite{monotoDn-cages} & Constr.~\ref{constr:odd} \pentaAT\\
5 & 13 & 8 & 378 & Th. 2.3 \cite{Aurajo-Pardo2007} & 458 & Th. 2.2 \cite{biregGirth8} & \\
5 & 13 & 9 & 1 106 & Moore bound & \underline{3 802} & Th. 2.1 \cite{monotoDn-cages} & \\
5 & 13 & 10 & 1 530 & Th. 2.3 \cite{Aurajo-Pardo2007} & 3 804 & Th. 2.2 (i) \cite{Aurajo-Pardo2007} & \\
5 & 13 & 11 & 4 434 & Moore bound & \underline{7 976} & Th. 2.1 \cite{monotoDn-cages} & \\
5 & 13 & 12 & 6 138 & Th. 2.3 \cite{Aurajo-Pardo2007} & 7 978 & Th. 2.2 \cite{biregGirth8} & \\
\bottomrule
\end{tabular}
}
\caption{Lower and upper bound for $n\rmg$ with $r=5$ and $m$ from 11 till 13.}
\end{table}

\end{document}